\documentclass[a4paper,12pt,fleqn]{smfart}
\usepackage[latin1]{inputenc}
\usepackage[francais,english]{babel} \AutoSpaceBeforeFDP
\usepackage[T1]{fontenc}
\usepackage{layout}
\usepackage{amstext,amssymb}
\usepackage{smfthm}
\usepackage{graphicx}
\usepackage{color}
\usepackage{epsfig}
\usepackage[all]{xy}
\usepackage{enumerate}
\usepackage{version}

\newcommand{\C}{\mathcal{C}}
\newcommand{\R}{\mathbb{R}}
\newcommand{\N}{\mathbb{N}}

\newcommand{\Q}{\mathbb{Q}}
\newcommand{\G}{\mathcal{G}}
\newcommand{\s}{\mathrm{SL_3(\mathbb{R})}}
\renewcommand{\ss}{\mathrm{SL_{n+1}(\mathbb{R})}}
\renewcommand{\C}{\mathcal{C}}
\newcommand{\Cc}{\mathrm{C}}
\renewcommand{\H}{\mathcal{H}}

\newcommand{\A}{\mathcal{A}}

\renewcommand{\O}{\Omega}
\newcommand{\g}{\gamma}
\renewcommand{\G}{\Gamma}

\renewcommand{\S}{\mathbb{S}}
\renewcommand{\P}{\mathbb{P}^2}

\newcommand{\PP}{\mathbb{P}}
\newcommand{\Aut}{\textrm{Aut}}

\newcommand{\Ax}{\textrm{Axe}}
\newcommand{\Hom}{\textrm{Hom}}

\newcommand{\Hol}{\textrm{Hol}}

\newcommand{\Quo}{\O/_{\G}}

\renewcommand{\Q}{\mathcal{Q}}

\theoremstyle{plain}
\newtheorem{fait}{Fait}

\theoremstyle{definition}

\theoremstyle{remark}
\newtheorem*{rem}{Remarque}

\newtheorem*{thm}{Théorème}

\title{Espace de Modules Marqués des Surfaces Projectives Convexes de Volume Fini}
\author{Ludovic Marquis}
\email{ludovic.marquis@umpa.ens-lyon.fr}
\urladdr{www.math.u-psud.fr/~marquis}

\begin{document}
\renewcommand{\labelitemi}{$\bullet$} %Change le symbole du 1er itemize

\maketitle

\begin{abstract}
Cet article est la suite de l'article \cite{ludo} dans lequel l'auteur caractérisait le fait d'être de volume fini pour une surface projective convexe. On montre ici que l'espace des modules $\beta_f(\Sigma_{g,p})$ des structures projectives convexes de volume fini sur la surface $\beta_f(\Sigma_{g,p})$ de genre $g$ à $p$ pointes est homéomorphe à $\R^{16g-16+6p}$.

Enfin, on montre que $\beta_f(\Sigma_{g,p})$ s'identifie à une composante connexe de l'espace des représentations du groupe fondamental de $\Sigma_{g,p}$ dans $\s$ qui conservent les paraboliques à conjugaison près.
\end{abstract}

\begin{altabstract}
This article follow the article \cite{ludo} in which the author characterize the fact of being of finite volume for a convex projective surface. We show here that the moduli space $\beta_f(\Sigma_{g,p})$ of convex projective structures on the surface $\Sigma_{g,p}$ of genius $g$ with $p$ punctures is homeomorphic to $\R^{16g-16+6p}$.

Finally, we show that $\beta_f(\Sigma_{g,p})$ can be identify with a connected component of the space of representation of the fundamental group of $\Sigma_{g,p}$ in $\s$ which keep the parabolic modulo conjugaison.
\end{altabstract}

%\tableofcontents

\mainmatter

\section{Introduction}

\subsection{Présentation des résultats}

\par{
Soit $\C$ une partie de l'espace projectif réel $\PP^n=\PP^n(\R)$, on dira que $\C$ est \emph{convexe} lorsque l'intersection de $\C$ avec toute droite de $\PP^n$ est connexe. Une partie convexe $\C$ est dite \emph{proprement convexe} lorsqu'il existe un ouvert affine contenant l'adhérence $\overline{\C}$ de $\C$.}
\\
\par{
Une structure projective proprement convexe sur une variété sans bord $M$ est un homéomorphisme entre $M$ et le quotient d'un ouvert proprement convexe $\Omega$ de $\PP^n$ par un sous-groupe discret de $\ss$ qui préserve $\Omega$. De tels structures sur les variétés compactes ont été étudiées ces dernières années par Benoist, Choi, Goldman, Labourie et Loftin \cite{Beno1,Beno2,Beno5,Beno6,Gold1,ChGo,Lab,Loft1}.
}
\\
\par{
L'ensemble des structures projectives proprement convexes de volume fini à isotopie près sur une surface à bord $S$ possède une topologie naturelle (on fera une hypothèse technique sur la géométrie du bord pour simplifer notre étude). Nous allons étudier l'espace des modules marqués des structures projectives proprement convexes à bord géodésique principal et de volume fini sur $S$, on notera cet espace $\beta_f(S)$. Goldman a montré que si $S$ est une surface sans bord, compact et de genre $g \geqslant 2$ alors l'espace $\beta_f(S)$ est homéomorphe à une boule de dimension $16g-16$. On notera $\Sigma_{g,p}$ la surface de genre $g$ à $p$ pointes. Nous allons montrer le théorème suivant:
}

\begin{thm}[\ref{modulesansbord}]\label{principal1}
Supposons que la surface $\Sigma_{g,p}$ soit de caractéristique d'Euler strictement négative alors l'espace des modules des structures projectives proprement convexes de volume fini sur la surface $\Sigma_{g,p}$ est homéomorphe à une boule de dimension $16g-16+6p$.
\end{thm}

\par{
Si $S$ est une surface à bord avec un bord non trivial alors l'espace $\beta_f(S)$ n'est pas une variété mais une variété à bord topologique dont nous décomposerons le bord en strates. Le système de coordonnées introduit sur l'espace $\beta_f(S)$ est une généralisation du système de coordonnées utilisé par Goldman pour étudier cet espace dans le cas où la surface $S$ est compacte sans bord. Le système de coordonnées de Goldman étant lui-même une généralisation du système de coordonnées utilisé par Fenchel et Nielsen pour étudier l'espace des modules des structures hyperboliques sur une surface.
}
\\
\par{
Dans l'article \cite{ludo}, l'auteur a donné des caractérisations du fait qu'une structure projective proprement convexe est de volume fini. On a notamment montré qu'une structure projective proprement convexe à bord géodésique sur une surface à bord de type fini $S$ est de volume fini si et seulement si l'holonomie des lacets élémentaires non homotopes à une composante connexe du bord de $S$ est parabolique.
}
\\
\par{
Ainsi, \cite{ludo} permet de faire le lien entre l'espace $\beta_f(\Sigma_{g,p})$ et un espace de représentations du groupe fondamental de $\Sigma_{g,p}$ étudié par Fock et Goncharov. En effet, l'holonomie  fournit une application de l'espace $\beta_f(\Sigma_{g,p})$ vers l'espace des représentations du groupe fondamental de $\Sigma_{g,p}$ suivant:
}

$$
\begin{array}{ccc}
\beta'_{g,p} & = & \left\{
                         \begin{tabular}{c|p{6cm}}
                                            & $\rho$ est fidèle et discrète,\\
$\rho \in \Hom(\pi_1(\Sigma_{g,p}),\s)$     & $\textrm{Im}(\rho)$ préserve un ouvert  proprement convexe $\Omega_{\rho}$,\\
                                            & le quotient $\Omega_{\rho}/_{\G}$ est homéomorphe à $\Sigma_{g,p}$,\\
                                            & l'holonomie des lacets élémentaires de $\pi_1(\Sigma_{g,p})$ est parabolique.
                           \end{tabular}
                   \right\}
\end{array}
$$

\par{
On verra que le groupe $\s$ agit librement et proprement sur $\beta'_{g,p}$, on note $\beta_{g,p}$ l'espace quotient. L'auteur a montré dans \cite{ludo}, que l'application naturelle donnée par l'holonomie de l'espace $\beta_f(\Sigma_{g,p})$ vers l'espace $\beta_{g,p}$ est injective. Nous montrerons que c'est en fait un homéomorphisme. Foch et Goncharov ont montré dans \cite{FoGo} que l'espace $\beta_{g,p}$ est homéomorphe à $\R^{16g-16+6p}$. Nous donnons ici un système de coordonnées "à la Fenchel-Nielsen" sur $\beta_f(\Sigma_{g,p})$ et nous gérons aussi le cas d'une surface à bord.
}
\\
\par{
Nous allons aussi montrer que les structures projectives proprement convexes de volume fini sont des objets naturels. Pour cela, nous montrerons le théorème suivant:
}

\begin{thm}[\ref{comptheo}]
On se donne une surface sans bord $S$ de caractéristique d'Euler strictement négative. L'espace des modules des structures projectives proprement convexe de volume fini sur la surface $S$ s'identifie à l'une des composantes connexes de l'espace des représentations du groupe fondamental de $S$ dans $\s$ dont l'holonomie des lacets élémentaires est parabolique, à conjugaison près.
\end{thm}

%\subsection{historique}
%
%\par{
%Il faut signaler que d'autres personnes se sont intéressées à l'espace $\beta_f(S)$ lorsque la surface $S$ est compacte. Le théorème suivant, obtenu de façon indépendante par Loftin et Labourie, est une amélioration du théorème de Goldman.
%}
%
%\begin{theo}[Loftin-Labourie]
%Soit $S$ une surface compacte, sans bord et de type fini et de caractéristique d'Euler strictement négative, il existe une fibration $\varphi:\beta_f(S) \rightarrow \textrm{Teich}(S)$ dont les fibres $\varphi^{-1}(\mathcal{S})$ s'identifient à l'espace vectoriel des différentielles cubiques holomorphes sur la surface de Riemann $\mathcal{S}$. De plus, l'action naturelle du mapping class group de $S$ sur  $\beta_f(S)$ préserve cette fibration.
%\end{theo}
%
%\\
%\par{
%Nous allons montrer un résultat analogue pour les surfaces de volume fini. On notera $\Sigma_{g,p}$ la surface de genre $g$ à $p$ pointes. On dira qu'un élément de $\pi_1(\Sigma_{g,p}) - \{ 1 \}$ est \emph{parabolique} lorsqu'il est représentable par un lacet librement homotopable dans tout voisinage de l'une des pointes de $\Sigma_{g,p}$. Cette définition ne dépend pas du choix d'un point base sur $\Sigma_{g,p}$. On introduit l'espace suivant:
%}
%
%
%
%
%\par{Il est naturel de s'intéresser à l'espace des modules suivant:}
\par{
Terminons cette introduction en donnant le plan de ce texte. La partie \ref{preliminaire} est un rappel des définitions de surfaces projectives, projectives proprement convexes, marquées, à bord géodésique, à bord géodésique principal. On définit ensuite la topologie de l'espace $\beta_f(S)$. On termine cette partie en donnant des conséquences faciles de la classification des automorphismes des ouverts proprement convexes (paragraphe \ref{consfaci}) qui nous seront utiles.
}
\\
\par{
Dans la troisième partie, on démontre les théorèmes \ref{modulesansbord} et \ref{module}. La démonstration se déroule en deux parties. On commence dans la partie \ref{point4} par étudier le recollement d'une surface projective proprement convexe le long d'un lacet. Ce résultat a été montré par Goldman, mais nous reproduisons une preuve pour la commodité du lecteur. Enfin, nous calculons l'espaces des modules des pantalons dont les classes de conjugaison de l'holonomie des bords est fixée. L'idée de départ est dû à Goldman, il s'agit de découper un pantalon à l'aide de deux triangles "en spirales". Cette idée permet de construire une bijection entre les pantalons projectifs proprement convexes et un objet combinatoire. La démonstration de Goldman de ce point nécessite de nouveaux arguments pour passer au cas non compact. Enfin, on termine cette partie en calculant, l'espace des modules de l'objet combinatoire. Cette partie est une généralisation facile de la preuve de Goldman.
}
\\
\par{
Dans la quatrième partie, on démontre le théorème \ref{comptheo}. La démonstration de ce théorème se fait en deux parties. Il faut montrer la fermeture et l'ouverture de l'espace $\beta_f(S)$ dans l'espace des représentations du groupe fondamental de $S$ dont l'holonomie d'un lacet élémentaire est parabolique. La démonstration de la fermeture est très proche du travail de Choi et Goldman qui montre dans \cite{ChGo}, la fermeture de $\beta_f(S)$ dans l'espace des représentations du groupe fondamental de $S$, lorsque $S$ est une surface compacte de genre supérieure ou égale à 2. Pour obtenir l'ouverture de $\beta_f(S)$ dans l'espace des représentations paraboliques, nous utiliserons le théorème d'invariance du domaine de Brouwer (le théorème \ref{modulesansbord} montre que $\beta_f(S)$ est une variété).
}

\section{Préliminaires}\label{preliminaire}

\subsection{Géométrie de Hilbert}\label{geohil2}

\subsubsection{Distance de Hilbert}

Cette partie constitue un rappel de faits connus en géométrie de
Hilbert. Un exposé plus complet peut être trouvé dans les
articles \cite{CVV1, CVV2, Beno3}.

Soit $\Omega$ un ouvert proprement convexe de $\PP^n$. Hilbert a
introduit sur de tels ouverts une distance, la distance de
Hilbert, définie de la façon suivante:

Soient $x \neq y \in \Omega$, on note $p,q$ les points d'intersection
de la droite $(xy)$ et du bord $\partial \Omega$ de $\Omega$ tels que $x$
est entre $p$ et $y$, et $y$ est entre $x$ et $q$ (Voir figure \ref{dis2}). On pose:

\begin{figure}[!h]
\begin{center}
\includegraphics[trim=0cm 12cm 0cm 0cm, clip=true, width=6cm]{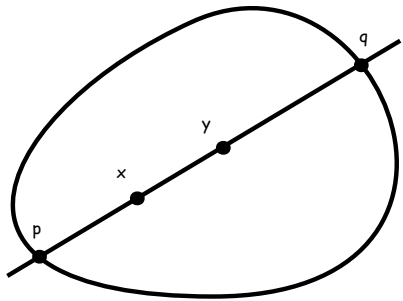}
%\centerline{\psfig{figure=dis.ps, width=6cm}}
\caption{La distance de Hilbert} \label{dis2}
\end{center}
\end{figure}

$$
\begin{array}{ccc}
d_{\Omega}(x,y) = \ln ( [p:x:y:q] ) = \ln \Big( \frac{\|p-y \|\cdot \|
q-x\|}{\| p-x \| \cdot \| q-y \|} \Big) & \textrm{et} &
d_{\Omega}(x,x)=0
\end{array}
$$
\begin{itemize}
\item $[p:x:y:q]$ désigne le birapport des points $p,x,y,q$.

\item $\| \cdot \|$ est une norme euclidienne quelconque sur un
ouvert affine $A$ qui contient $\overline{\Omega}$.
\end{itemize}

\begin{rem}
Il est clair que $d_{\Omega}$ ne dépend ni du choix de $A$, ni du
choix de la norme euclidienne sur $A$.
\end{rem}

\begin{fait}
Soit $\Omega$ un ouvert proprement convexe de $\PP^n$.
\begin{itemize}
\item $d_{\Omega}$ est une distance sur $\Omega$.

\item $(\Omega,d_{\Omega})$ est un espace métrique complet.

\item La topologie induite par $d_{\Omega}$ coïncide avec celle
induite par $\PP^n$.

\item Le groupe $\Aut(\Omega)$ des transformations projectives de $\ss$ qui préservent $\Omega$ est un sous-groupe fermé de $\ss$ qui agit par isométrie sur $(\Omega,d_{\Omega})$. Il agit donc proprement sur $\Omega$.
\end{itemize}
\end{fait}

On peut trouver une démonstration de cet énoncé dans \cite{Beno3}.

\subsubsection{La structure finslérienne d'un ouvert proprement convexe}

Soit $\Omega$ un ouvert proprement convexe de $\PP^n$. La métrique de
Hilbert $d_{\Omega}$ est induite par une structure finslérienne sur
l'ouvert $\Omega$. On identifie le fibré tangent $T \Omega$ de $\Omega$ à
$\Omega \times A$.

Soient $x \in \Omega$ et $v \in A$. On note $p^+$ (resp $p^-$) le
point d'intersection de la demi-droite définie par $x$ et $v$
(resp $-v$) avec $\partial \Omega$.

On pose: $\|v\|_x = \Big( \frac{1}{\|x-p^-\|} + \frac{1}{\| x-
p^+\|} \Big) \| v \|$.

\begin{figure}[!h]
\begin{center}
\includegraphics[trim=0cm 12cm 0cm 0cm, clip=true, width=6cm]{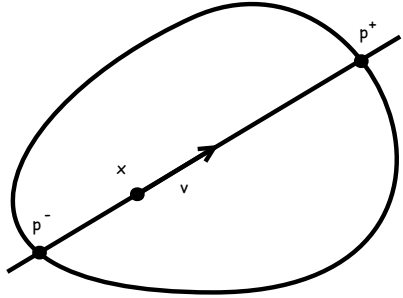}
\caption{La métrique de Hilbert} \label{met2}
\end{center}
\end{figure}

\begin{fait}
Soient $\Omega$ un ouvert proprement convexe de $\P$ et $A$ un ouvert
affine qui contient $\overline{\Omega}$.
\begin{itemize}
\item La distance induite par la métrique finslérienne $\|\cdot \|
_{\cdot}$ est la distance $d_{\Omega}$.

\item Autrement dit on a les formules suivantes:
\begin{itemize}
\item $\displaystyle{\|v\|_x = \left. \frac{d}{dt}  \right| _{t=0} d_{\Omega}(x,x+tv)}$, où $v \in
A$, $t\in \R$ assez petit.

\item $\displaystyle{d_{\Omega}(x,y) = \inf \int_0^1 \| \sigma'(t) \|_{\sigma(t)}
dt}$, où $l'\inf$ est pris sur les chemins $\sigma$ de classe
$\C^1$ tel que $\sigma(0)=x$ et $\sigma(1)=y$.
\end{itemize}
\end{itemize}
\end{fait}

\begin{rem}
$\|v\|_x$ est donc indépendante du choix de $A$ et de $\| \cdot
\|_{A}$.
\end{rem}

\subsubsection{Mesure sur un ouvert proprement convexe ( dite mesure de Busemann )}

Nous allons construire une mesure borélienne $\mu_{\Omega}$ sur $\Omega$,
de la même façon que l'on construit une mesure borélienne sur une
variété riemanienne.

Soit $\Omega$ un ouvert proprement convexe de $\PP^n$. On note:
\begin{itemize}
\item $B_x(1) = \{ v \in T_x \Omega \, | \, \|v\|_x < 1 \}$

\item $\textrm{Vol}$ est la mesure de Lebesgue sur $A$ normalisée
pour avoir $\textrm{Vol}(\{ v \in A \, | \, \|v\| < 1 \})=1$.
\end{itemize}

On peut à présent définir la mesure $\mu_{\Omega}$. Pour tout borélien
$\mathcal{A} \subset \Omega \subset A$, on pose:

$$\mu_{\Omega} (\mathcal{A})= \int_{\mathcal{A}} \frac{dVol(x)}{\textrm{Vol}(B_x(1))}$$

La mesure $\mu_{\Omega}$ est indépendante du choix de $A$ et de $\| \cdot \|$,
car c'est la mesure de Haussdorff de $(\Omega,d_{\Omega})$ (Exemple 5.5.13
\cite{BBI}). ( Pour une introduction aux mesures de Haussdorff, on
pourra regarder \cite{BBI}). La mesure $\mu_{\Omega}$ est donc
$\Aut(\Omega)$-invariante.

\subsection{Dynamique des éléments de $\Aut(\Omega)$}\label{classi_auto}

Dans ce paragraphe nous allons rappeller certaines propositions nécessaires pour la suite de ce texte. Toutes ses propositions sont démontrés dans l'article \cite{ludo}.

\begin{prop}\label{classi2}
Soit $\Omega$ un ouvert proprement convexe. Soit $\g \in \Aut(\Omega)$,
$\g$ fait partie de l'une des six familles
suivantes:

\begin{itemize}
\item La famille des éléments dits $\underline{hyperboliques}$ qui sont conjugués à une matrice de la forme suivante:

$$
\begin{array}{cc}
\left(
\begin{array}{ccc}
\lambda^+ &    0      & 0  \\
 0        & \lambda^0 & 0  \\
 0        &    0      & \lambda^-\\
\end{array}
\right)
&
\begin{tabular}{l}
où, $\lambda^+ > \lambda^0 > \lambda^- > 0$\\
et $\lambda^+ \lambda^0 \lambda^- = 1$.
\end{tabular}
\end{array}
$$

\item La famille des éléments dits $\underline{planaires}$ qui sont conjugués à une matrice de la forme suivante:

$$
\begin{array}{cc}
\left(
\begin{array}{ccc}
\alpha &   0        & 0   \\
 0     & \alpha     & 0    \\
 0     &   0        & \beta\\
\end{array}
\right) &
\begin{tabular}{l}
où, $\alpha, \beta > 0$, $\alpha^2 \beta = 1$\\
et $\alpha,\, \beta \neq 1$.
\end{tabular}
\end{array}
$$

\item La famille des éléments dits
$\underline{quasi-hyperboliques}$ qui sont conjugués à une matrice de la forme suivante:

$$
\begin{array}{cc}
\left(
\begin{array}{ccc}
\alpha &  1     &  0    \\
 0     & \alpha &  0   \\
 0     &  0     & \beta\\
\end{array}
\right)
 &
\begin{tabular}{l}
où, $\alpha, \, \beta > 0$, $\alpha^2 \beta = 1$\\
et $\alpha, \,\beta \neq 1$.
\end{tabular}
\end{array}
$$

\item La famille des eléments dits $\underline{paraboliques}$ qui sont conjugués à la matrice suivante:

$$
\left(
\begin{array}{ccc}
1 & 1 & 0\\
0 & 1 & 1\\
0 & 0 & 1\\
\end{array}
\right)
$$

\item La famille des eléments dits \underline{elliptiques} qui sont conjugués à une matrice de la forme suivante:

$$
\begin{tabular}{cc}
$\left(
\begin{array}{ccc}
1      &   0          & 0\\
0      & \cos(\theta) & -\sin(\theta)\\
0      & \sin(\theta) & \cos(\theta)\\
\end{array}
\right)$ &

Où, $0 < \theta < 2\pi$.

\end{tabular}
$$

\item La famille composée uniquemenent de \underline{L'identité}:
\end{itemize}
\end{prop}

\subsubsection{Dynamique hyperbolique}

Soit $\g$ un élément hyperbolique de $\s$, $\g$ est conjugué à la
matrice

$$
\begin{array}{cc}
\left(
\begin{array}{ccc}
\lambda^+ &   0       & 0  \\
 0        & \lambda^0 & 0  \\
 0        &   0       & \lambda^-\\
\end{array}
\right) &
\begin{tabular}{l}
Où, $\lambda^+ > \lambda^0 > \lambda^- > 0$\\
et $\lambda^+ \lambda^0 \lambda^- = 1$.
\end{tabular}
\end{array}
$$

On note:
\begin{itemize}
\item $p^+_{\g}$ le point propre de $\P$ associé à la valeur
propre $\lambda^+$.

\item $p^0_{\g}$ le point propre de $\P$ associé à la valeur
propre $\lambda^0$.

\item $p^-_{\g}$ le point propre de $\P$ associé à la valeur
propre $\lambda^-$.

\item $D^{+,-}_{\g}$ la droite stable de $\P$ associée aux valeurs
propres $\lambda^+,\lambda^-$.

\item $D^{+,0}_{\g}$ la droite stable de $\P$ associée aux valeurs
propres $\lambda^+,\lambda^0$.

\item $D^{-,0}_{\g}$ la droite stable de $\P$ associée aux valeurs
propres $\lambda^-,\lambda^0$.
\end{itemize}

\begin{rem}
Tout au long de ce texte, l'indice $\g$ pour désigner un espace
stable de $\g$ sera omis si le contexte est clair.
\end{rem}

\begin{prop}
Soient $\Omega$ un ouvert proprement convexe et $\g \in \Aut(\Omega)$ un
élément hyperbolique. Alors, on a $p^+ ,\, p^- \in \partial \Omega$ et $p^0 \notin \Omega$.
\end{prop}

On peut à présent définir l'axe d'un élément hyperbolique qui agit sur un ouvert proprement convexe.

\begin{defi}
Soient $\Omega$ un ouvert proprement convexe et $\g \in \Aut(\Omega)$ un
élément hyperbolique. \emph{L'axe de $\g$} que l'on notera $\Ax(\g)$ est le segment ouvert de la droite $(p^- p^+)$ qui est inclus dans $\overline{\Omega}$ et dont les extrémités sont $p^-$ et $p^+$.
\end{defi}

La figure \ref{fhyp2} illustre la dynamique d'un élément hyperbolique.

\begin{figure}[!h]
\begin{center}
\includegraphics[trim=0cm 7cm 0cm 0cm, clip=true, width=6cm]{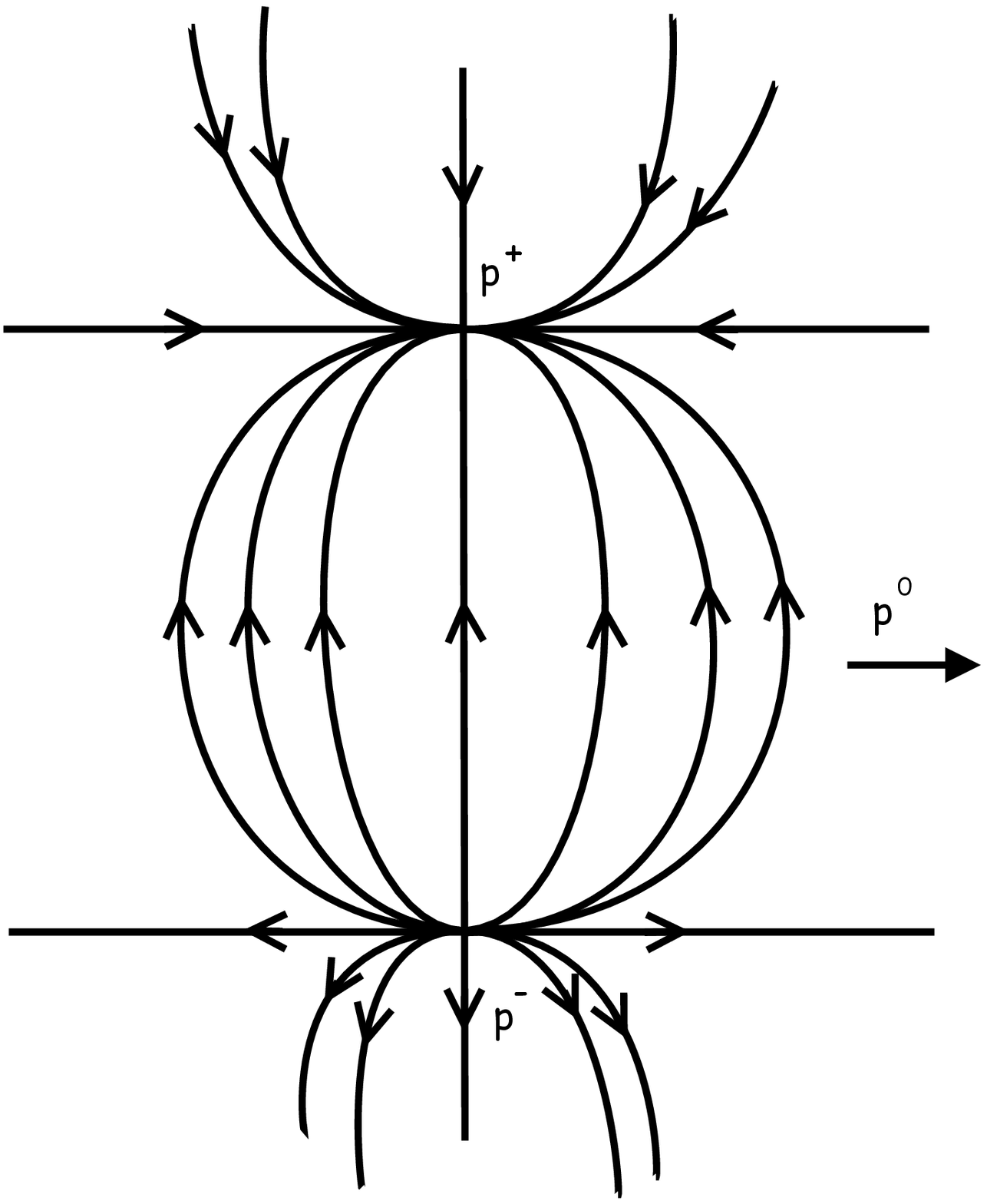}
\label{fhyp2}\caption{Dynamique d'un élément hyperbolique}
\end{center}
\end{figure}

\begin{prop}\label{hyp2}
Soient $\Omega$ un ouvert proprement convexe et $\g \in \Aut(\Omega)$ un
élément hyperbolique.
\begin{itemize}
\item Si $p^0 \notin \overline{\Omega}$ et $\Ax(\g) \subset \Omega$
alors
      \begin{itemize}
      \item $\partial \Omega$ est $\Cc^1$ en $p^+$ et en $p^-$.

      \item $p^0 = T_{p^+} \partial \Omega \cap T_{p^-} \partial \Omega$.
      \end{itemize}

\item Si $p^0 \notin \overline{\Omega}$ et $\Ax(\g) \subset
\partial \Omega$ alors
      \begin{itemize}
      \item $\partial \Omega$ n'est pas $\Cc^1$ en $p^+$ et en $p^-$.
      \item Les demi-tangentes à $\partial \Omega$ en $p^+$ sont $(p^+ p^-)$ et $(p^+ p^0)$.
      \item Les demi-tangentes à $\partial \Omega$ en $p^-$ sont $(p^+ p^-)$ et $(p^- p^0)$.
      \end{itemize}

\item Si $p^0 \in \overline{\Omega}$ et $\Ax(\g) \subset \Omega$ alors
      \begin{itemize}
      \item $[p^+,p^0]$ et $[p^-,p^0] \subset \partial \Omega$.
      \item $\partial \Omega$ est $\Cc^1$ en $p^+$, $p^-$ et $p^0 = T_{p^+} \partial \Omega \cap T_{p^-} \partial
      \Omega$.
      \end{itemize}

\item Si $p^0 \in \overline{\Omega}$ et $\Ax(\g) \subset \partial
\Omega$ alors
      \begin{itemize}
      \item $[p^+,p^0]$ et $[p^-,p^0] \subset \partial \Omega$.
      \item $\Omega$ est un
triangle dont les sommets sont $p^+, \, p^0, \, p^-$.
      \end{itemize}
\end{itemize}
\end{prop}

\subsubsection{Dynamique planaire}

Soit $\g$ un élément planaire de $\s$, $\g$ est conjugué à la
matrice

$$
\begin{array}{cc}
\left(
\begin{array}{ccc}
\alpha &  0        & 0   \\
 0     & \alpha    & 0   \\
 0     &  0        & \beta\\
\end{array}
\right) &
\begin{tabular}{l}
où, $\alpha, \beta > 0$, $\alpha^2 \beta = 1$\\
et $\alpha,\, \beta \neq 1$.
\end{tabular}
\end{array}
$$

On note:
\begin{itemize}
\item $p_{\g}$ le point propre de $\P$ associé à la valeur propre
$\beta$.

\item $D_{\g}$ la droite propre de $\P$ associée à la valeur
propre $\alpha$.
\end{itemize}

La dynamique des éléments planaires étant extrément simple, on
obtient facilement la propostion suivante.

\begin{prop}\label{planaire2}
Soit $\Omega$ un ouvert proprement convexe. Soit $\g \in \Aut(\Omega)$ un
élément planaire. Alors, $\Omega$ est un triangle dont l'un des
sommets est $p_{\g}$ et le côté de $\Omega$ opposé à $p_{\g}$ est
inclus dans la droite $D_{\g}$.
\end{prop}

Comme le stabilisateur d'un triangle dans $\P$ est le groupe des matrices diagonales à diagonale strictement positive. On obtient le corollaire suivant:

\begin{coro}
Soient $\Omega$ un ouvert proprement convexe et $\G$ un sous-groupe discret sans torsion de $\Aut(\Omega)$. Si la surface $\Quo$ est de caractéristique d'Euler strictement négative alors $\Omega$ n'est pas un triangle. En particulier, aucun élément du groupe $\G$ n'est planaire.
\end{coro}

\subsubsection{Dynamique quasi-hyperbolique}

Soit $\g$ un élément quasi-hyperbolique de $\s$, $\g$ est conjugué
à la matrice

$$
\begin{array}{cc}
\left(
\begin{array}{ccc}
\alpha &  1     &  0   \\
 0     & \alpha &  0   \\
 0     &  0     & \beta\\
\end{array}
\right)
 &
\begin{tabular}{l}
Où, $\alpha, \, \beta > 0$, $\alpha^2 \beta = 1$\\
et $\alpha, \,\beta \neq 1$.
\end{tabular}
\end{array}
$$

On note:
\begin{itemize}
\item $p^1_{\g}$ le point propre de $\P$ associé à la valeur
propre $\beta$.

\item $p^2_{\g}$ le point propre de $\P$ associé à la valeur
propre $\alpha$.

\item $D_{\g}$ la droite stable de $\P$ associée à la valeur
propre $\alpha$.
\end{itemize}

On peut définir l'axe d'un élément quasi-hyperbolique qui agit sur un ouvert proprement convexe de la même façon que pour un élément hyperbolique. La même démonstration que dans le cas hyperbolique donne la proposition suivante:

\begin{prop}
Soient $\Omega$ un ouvert proprement convexe et $\g \in \Aut(\Omega)$ un
élément quasi-hyperbolique. Alors, on a $p^1 ,\, p^2 \in \partial \Omega$.
\end{prop}

\begin{defi}
Soient $\Omega$ un ouvert proprement convexe et $\g \in \Aut(\Omega)$ un
élément quasi-hyperbolique.\emph{ L'axe de $\g$} que l'on notera $\Ax(\g)$ est le segment ouvert de la droite $(p^1 p^2)$ qui est inclus dans $\overline{\Omega}$ et dont les extrémités sont $p^1$ et $p^2$.
\end{defi}

\begin{figure}[!h]
\begin{center}
\includegraphics[trim=0cm 12cm 0cm 0cm, clip=true, width=6cm]{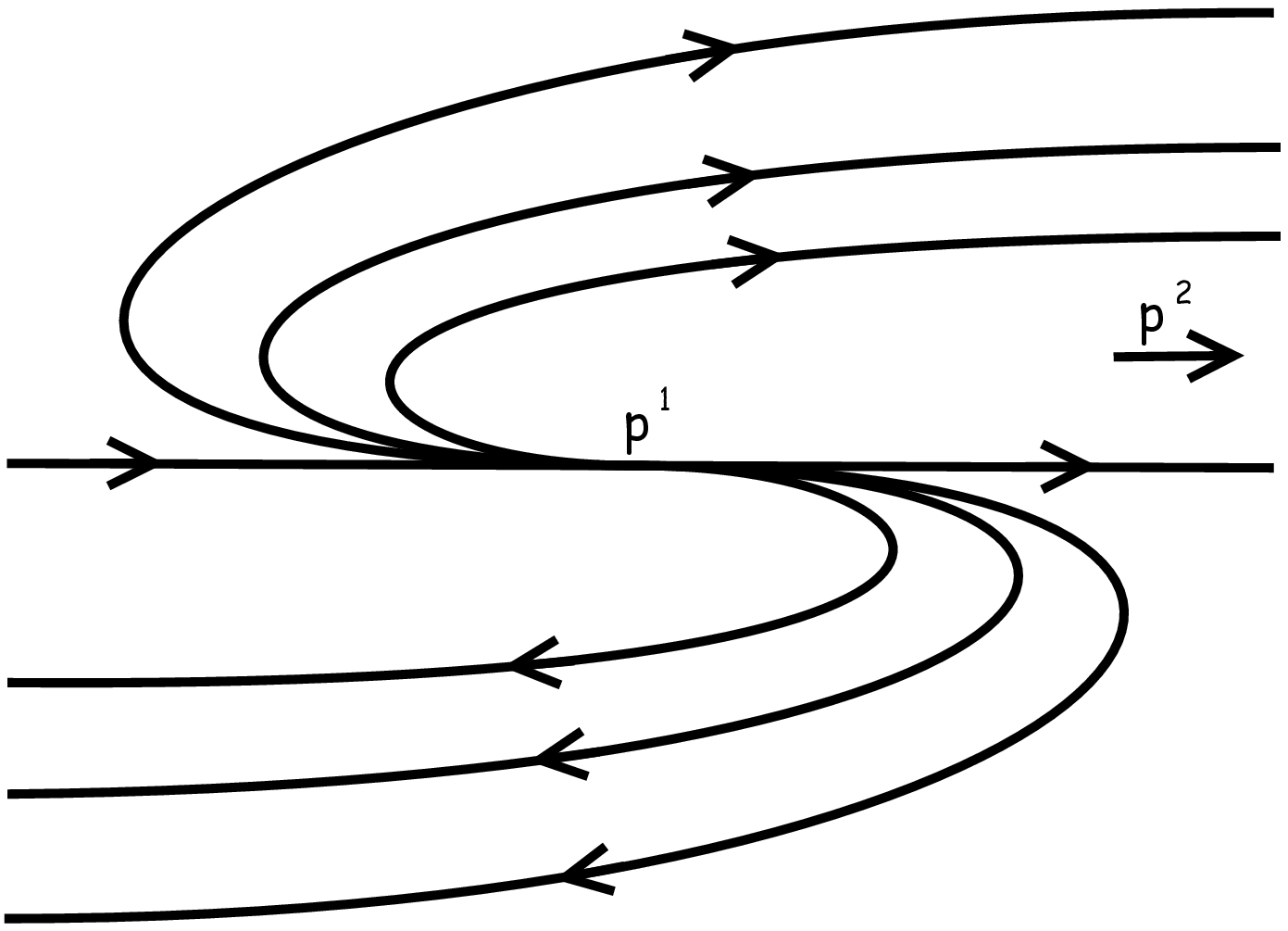}
\label{qhyp2}\caption{Dynamique d'un élément quasi-hyperbolique}
\end{center}
\end{figure}

\begin{prop}\label{quasihyp2}
Soit $\Omega$ un ouvert proprement convexe. Soit $\g \in \Aut(\Omega)$ un
élément quasi-hyperbolique. Alors,

\begin{itemize}
\item $\Ax(\g) \subset \partial \Omega$.

\item $\partial \Omega$ n'est pas $\Cc^1$ en $p^2$ et les
demi-tangentes à $\partial \Omega$ en $p^2$ sont $(p^1 p^2)$ et $D$.

\item $\partial \Omega$ est $\Cc^1$ en $p^1$ et $T_{p^1} \partial \Omega =
(p^1 p^2)$.
\end{itemize}
\end{prop}

\subsubsection{Dynamique parabolique}

Soit $\g$ un élément parabolique de $\s$, $\g$ est conjugué à la
matrice

$$
\left(
\begin{array}{ccc}
1  & 1    & 0 \\
0  & 1    & 1\\
0  & 0    & 1\\
\end{array}
\right)
$$

On note:
\begin{itemize}
\item $p_{\g}$ l'unique point fixe de $\g$ sur $\P$.

\item $D_{\g}$ l'unique droite fixe de $\g$ sur $\P$.
\end{itemize}

La figure \ref{para2} illustre la dynamique d'un élément parabolique.

\begin{figure}[!h]
\begin{center}
\includegraphics[trim=0cm 7cm 0cm 0cm, clip=true, width=6cm]{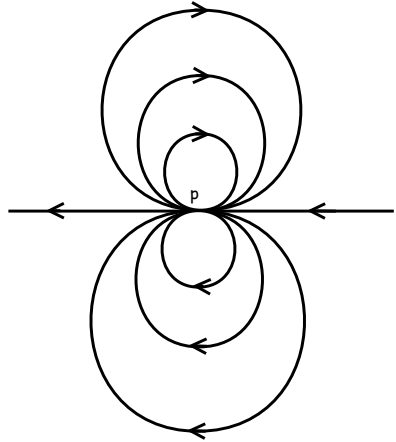}
\label{para2}\caption{Dynamique d'un élément parabolique}
\end{center}
\end{figure}

\begin{prop}\label{ppara2}
Soit $\Omega$ un ouvert proprement convexe. Soit $\g \in \Aut(\Omega)$ un
élément parabolique. Alors,

\begin{itemize}
\item $p \in \partial \Omega$.

\item $\partial \Omega$ est $\Cc^1$ en $p$.

\item $T_p \partial \Omega = D$.

\item $p$ n'appartient pas à un segment non trivial du bord de $\Omega$.
\end{itemize}
\end{prop}

\subsubsection{Dynamique elliptique}

On considère $\G$ un sous-groupe discret de $\s$ qui préserve un ouvert proprement convexe de $\P$ et on suppose que $\G$ est un groupe de surface.

Comme le groupe $\G$ est un sous-groupe discret de $\s$, tout éventuel élément elliptique de $\G$ est d'ordre fini. Mais le groupe $\G$ est un groupe de surface donc sans torsion. Aucun élément du groupe $\G$ n'est donc elliptique.

\subsection{Conséquence facile de la classification des automorphismes des ouverts proprement convexes}\label{consfaci}

Les propositions suivantes sont des conséquences évidentes des propositions du paragraphe \ref{classi_auto}.

\begin{prop}\label{tra}
Soient $\Omega$ un ouvert proprement convexe et $\g \in \Aut(\Omega)$ d'ordre infini, alors, la trace de $\g$ est supérieure ou égale à 3 avec égalité si et seulement si $\g$ est parabolique ou l'élément identité.
\end{prop}

On introduit les sous-ensembles de $\R^2$ suivants:

$$
\begin{array}{lcl}
\mathfrak{R} & = & \{ (\lambda,\tau) \in \R^2 \,|\,  0 < \lambda < 1, \, \frac{2}{\sqrt{\lambda}} < \tau < \lambda + \lambda^{-2} \}\\

\mathfrak{R}_{QH\,non\, C^1} & = & \{ (\lambda,\tau) \in \R^2 \,|\, 0 < \lambda < 1, \, \tau = \lambda + \lambda^{-2} \}\\

\mathfrak{R}_{QH \, C^1} & = & \{ (\lambda,\tau) \in \R^2 \,|\, 0 < \lambda < 1, \, \tau = \frac{2}{\sqrt{\lambda}} \}\\

\mathfrak{R}_{P} & = & \{ ( 1,2) \}\\

\widehat{\mathfrak{R}} & = & \mathfrak{R} \cup \mathfrak{R}_{QH\,non\, C^1} \cup \mathfrak{R}_{QH \, C^1}
\end{array}
$$

\par{
On pourra remarquer que $\mathfrak{R}$ est un ouvert de $\R^2$ homéomorphe à $\R^2$ dont le bord est la réunion des 2 sous-variétés $\mathfrak{R}_{QH\,non\, C^1}$ et $\mathfrak{R}_{QH \, C^1}$ homéomorphes à $\R$ et du point $\mathfrak{R}_{P}$.
}
\\
\par{
Enfin, l'espace $\widehat{\mathfrak{R}}$ est une variété à bord homéomorphe à $[0,1] \times \R$.
}

\begin{defi}
Soit $\g \in \s$ dont le spectre est réel et strictement positif, on note $\lambda_1 \leqslant \lambda_2 \leqslant \lambda_3$ les valeurs propres de $\g$. On notera $\lambda(\g) = \lambda_1$ la plus petite valeur propre de $\g$ et $\tau(\g) = \lambda_2 + \lambda_3$ la somme des deux autres.
\end{defi}

La proposition suivante est alors une simple conséquence de la classification des automorphismes des ouverts proprement convexes.

\begin{prop}\label{consclassi}
Soient $\Omega$ un ouvert proprement convexe de $\P$ et $\g \in \Aut(\Omega)$, on suppose que $\g$ possède un point fixe $p \in \partial \Omega$. On suppose de plus que si $\g$ est quasi-hyperbolique ou hyperbolique alors $p$ est un point fixe répulsif. Alors le spectre de $\g$ est réel, strictement positif et on a:
\begin{itemize}
\item L'élément $\g$ est hyperbolique si et seulement si $(\lambda(\g),\tau(\g)) \in \mathfrak{R}$

\item L'élément $\g$ est quasi-hyperbolique et $p$ n'est pas un point $C^1$ de $\partial \Omega$ si et seulement si $(\lambda(\g),\tau(\g)) \in \mathfrak{R}_{QH\, non\, C^1}$

\item L'élément $\g$ est quasi-hyperbolique et $p$ est un point $C^1$ de $\partial \Omega$ si et seulement si $(\lambda(\g),\tau(\g)) \in \mathfrak{R}_{QH\, C^1}$

\item L'élément $\g$ est parabolique si et seulement si $(\lambda(\g),\tau(\g)) \in \mathfrak{R}_P$
\end{itemize}
\end{prop}

\subsection{Définition des surfaces projectives proprement convexes}

\par{
Tout au long de ce texte une \emph{surface} sera une variété connexe orientable de dimension 2 à bord. On note $\mathbb{E}$ un demi-espace affine fermé de $\P$.
}
\\
\par{
Une \emph{structure projective réelle à bord géodésique} sur une surface $S$ est la donnée d'un atlas maximal $\varphi_{\mathcal{U}}:\mathcal{U} \rightarrow \mathbb{E}$ sur $S$ tel que les fonctions de transitions $\varphi_{\mathcal{U}} \circ \varphi_{\mathcal{V}}^{-1}$ sont des éléments de $\s$, pour tous
ouverts $\mathcal{U}$ et $\mathcal{V}$ de l'atlas de $S$ tel que
$\mathcal{U} \cap \mathcal{V} \neq \varnothing$.
}

\begin{rem}
Pour simplifier la rédaction on dira \og structure projective\fg $\,$ à la place de \og structure projective réelle à bord géodésique \fg.
\end{rem}

\par{
Un \emph{isomorphisme} entre deux surfaces munies de structures projectives est un homéomorphisme qui, lu dans les cartes, est
donné par des éléments de $\s$.
}
\\
\par{
La donnée d'une structure projective sur une surface $S$ est équivalente à la donnée:
 \begin{enumerate}
 \item d'un homéomorphisme local $dev:\widetilde{S}\rightarrow \P$ appelée \emph{développante}, où $\widetilde{S}$ est le revêtement universel de $S$,

 \item d'une représentation $hol:\pi_1(S) \rightarrow \s$ appelée \emph{holonomie} tel que la développante est $\pi_1(S)$-équivariante ( i.e pour tout $x \in \widetilde{S}$, et pour tout $\g \in \pi_1(S)$ on a $dev(\g \, x) = hol(\g) dev(x)$).
\end{enumerate}
De plus, deux structures projectives données par les couples $(dev,hol)$ et $(dev',hol')$ sont isomorphes si et seulement si il existe un élément $g \in \s$ tel que $dev' = g \circ dev$ et $hol' = g \circ hol \circ g^{-1}$.
}
\\
\par{
Lorsqu'on s'intéresse à l'ensemble des structures sur une surface à équivalence près, il faut faire attention à introduire la bonne relation d'équivalence sur l'ensemble des structures en question. A première vue, on pourrait penser que la bonne relation d'équivalence est tout simplement l'existence d'un isomorphisme. Mais, si on introduit cette relation, alors l'espace quotient n'est en général pas une variété. C'est pourquoi on introduit la notion de surface marquée et la notion d'isotopie.
}

\begin{defi}
Soit $S$ une surface, une \emph{structure projective marquée sur $S$} est la donnée d'un homéomorphisme $\varphi:S \rightarrow S'$ où
$S'$ est une surface projective. On note $\PP'(S)$ \emph{l'ensemble des structures projectives marquées
sur $S$}.

Deux structures projectives marquées sur $S$, $\varphi_1:S \rightarrow S_1$ et $\varphi_2:S \rightarrow S_2$ sont dites
\emph{isotopiques} lorsqu'il existe un isomorphisme $h:S_1 \rightarrow S_2$ tel que $\varphi_2^{-1} \circ h \circ \varphi_1 : S
\rightarrow S$ est un homéomorphisme isotope à l'identité. On
note $\PP(S)$ \emph{l'ensemble des structures projectives marquées
sur $S$ modulo isotopie}.
\end{defi}

On peut à présent définir une topologie sur l'ensemble des structures projectives marquées sur la surface $S$. On introduit l'espace:

$$
\mathcal{D'}(S)=
\left\{
\begin{array}{l|l}
          & dev:\widetilde{S}\rightarrow \P  \textrm{ est un homéomorphisme local}\\
(dev,hol) & hol:\pi_1(S) \rightarrow \s\\
          & dev \textrm{ est } \pi_1(S)-\textrm{équivariante}
\end{array}
\right\}
$$

\par{
Les espaces $\widetilde{S}$, $\P$, $\pi_1(S)$ et $\s$ sont des espaces topologiques localement compacts. On munit l'ensemble des applications continues entre deux espaces localement compacts de la topologie compact-ouvert. Ainsi, l'espace $\mathcal{D}'(S)$ est munie d'une topologie. Le groupe $Homeo_0(S)$ des homéomorphismes isotopes à l'identité agit naturellement sur $\mathcal{D}'(S)$. Le groupe $\s$ agit aussi naturellement sur $\mathcal{D}'(S)$. Ces deux actions commutent. L'espace quotient est l'espace $\PP(S)$ des structures projectives marquées sur $S$ à isotopie près. On le munit de la topologie quotient.
}
\\
\par{
On ne s'intéresse qu'à un certain type de structure projective: les structures projectives proprement convexes.
}
\begin{defi}\label{def}
Soit $S$ une surface, une structure projective est dite
\emph{proprement convexe sur $S$} lorsque la développante est un
homéomorphisme sur une partie $\C$ proprement convexe de $\P$. On note
$\beta(S)$ l'ensemble des structures projectives proprement
convexes sur $S$ modulo isotopie.
\end{defi}

\par{
Il est important de remarquer que la définition que l'on a choisi de surfaces projectives fournit une contrainte sur la géométrie du bord d'une surface projective. En effet, on a supposé que l'atlas d'une surface projective était à valeur dans un demi-espace affine fermée $\mathbb{E}$. Par conséquent, pour tout ouvert $\mathcal{U}$ de l'atlas tel que $\mathcal{U} \cap \partial S \neq \varnothing$ et pour toute composante connexe $B$ de $\mathcal{U} \cap \partial S$, $\varphi_{\mathcal{U}}(B)$ est inclus dans une droite de $\P$. Ainsi, tout relevé d'une composante connexe $\mathcal{B}$ du bord de $S$ est un segment ouvert $s$ de $\P$ préservé par un élément $\g \in \s$, qui est l'holonomie d'un lacet librement homotope à la composante connexe $\mathcal{B}$ du bord de $S$. Ceci explique la terminologie de structure projective \og à bord géodésique \fg. Nous étudierons une classe légèrement plus restrictive de structure projective, on dira que le bord $\mathcal{B}$ est \emph{principal} lorsque $\g$ est un élément hyperbolique ou quasi-hyperbolique de $\s$ et l'une des extrémités de $s$ est un point attractif de $\g$ et l'autre un point répulsif. Une structure projective sera dite \emph{à bord géodésique principal} lorsque tout ces bords sont principaux.
}
\\
\par{
Soit $S$ une surface projective proprement convexe, l'application développante  permet d'identifier le revêtement universel $\widetilde{S}$ de $S$ à une partie $\C$ proprement convexe de $\P$. On notera $\pi:\C \rightarrow S$ le revêtement universel de $S$. L'intérieur $\Omega$ de $\C$ est naturellement muni d'une mesure $\mu_{\Omega}$ invariante sous l'action du groupe fondamental $\pi_1(S)$ de $S$. Par conséquent, il existe une unique mesure $\mu_S$ sur l'intérieur $\overset{\circ}{S}$ de $S$ telle que pour tout borélien $\mathcal{A}$ de $\Omega$, si $\pi:\Omega \rightarrow \overset{\circ}{S}$ restreinte à $\mathcal{A}$ est injective alors $\mu_S(\pi(\mathcal{A}))=\mu_{\Omega}(\mathcal{A})$.
}
\begin{defi}
Soit $S$ une surface, on dit qu'une structure projective
proprement convexe sur $S$ est de \emph{volume fini} lorsque pour
tout fermé $F$ de l'intérieur $\overset{\circ}{S}$ de $S$, on a
$\mu_S(F) < \infty$. On note $\beta_f(S)$ l'espace des modules des structures projectives marquées proprement convexes à bord géodésique principal de volume fini sur $S$.
\end{defi}

\section{Paramétrisation de l'espace des modules}

Dans ce paragraphe, nous allons expliquer la méthode que nous allons suivre pour montrer le théorème \ref{principal1}, ainsi que sa version dans le cadre des surfaces à bord. On termine en donnant les théorèmes de paramétrisation de l'espace des modules des structures projectives proprement convexes de volume fini sur une surface sans bord et à bord. Nous allons à présent rappeler quelques définitions élémentaires de topologie des surfaces.

\begin{defi}\label{elem}
Soit $S$ une surface, on dit qu'un lacet tracé sur $S$, $c:\S^1
\rightarrow S$ est \emph{simple} s'il est injectif. On dit qu'un
lacet simple $c$ tracé sur $S$ est \emph{élémentaire} si $S-c$
possède deux composantes connexes et l'une d'elles est un
cylindre. Lorsque $S$ n'est pas un cylindre, on appellera l'adhérence de la composante homéomorphe à un cylindre
\emph{la composante élémentaire associée à $c$}.
\end{defi}

\begin{rem}
Soient $S$ une surface et $c$ un lacet élémentaire tracé sur $S$,
alors on a l'alternative suivante:
\begin{itemize}
\item La composante élémentaire associée à $c$ est un cylindre
avec un seul bord. On dira alors que $c$ \emph{fait le tour d'un bout}.

\item La composante élémentaire associée à $c$ est un cylindre
avec deux bords. Dans ce cas $c$ est librement homotope à une
composante connexe du bord de $S$.
\end{itemize}
\end{rem}

Tout au long de ce texte, on notera $\Sigma_{g,b,p}$ la surface obtenue en retirant $b$ disques ouverts disjoints et $p$ points distincts à la surface de genre $g$ (on suppose bien sûr aussi que les points n'appartiennent pas à l'adhérence des disques que l'on retirent).

%$$
%\begin{array}{ccc}
%\beta_{Para}(\pi_1(\Sigma_{g,b,p})) & = & \left\{
%                         \begin{array}{c|l}
%                                                                 & \rho \textrm{ est fidèle},\\
%                          & \textrm{Im}(\rho) \textrm{ est discrète},\\
%                                                                 & \textrm{Im}(\rho) \textrm{ préserve un ouvert  proprement convexe},\\
%\rho \in \Hom(\pi_1(\Sigma_{g,p}),\ss)                                                                  & \textrm{le quotient } \Omega_{\rho}/_{\G} \textrm{ est homéomorphe à } \Sigma_{g,b+p},\\
%                                                                 & \textrm{l'holonomie des lacets de} \Sigma_{g,b,p}\\
%                                                                 & \textrm{qui font le tour d'un bout est parabolique},\\
%                                                                 & \textrm{l'holonomie des lacets de} \Sigma_{g,b,p}\\
%                                                                 & \textrm{homotope à une composante connexe du bord}\\
%                                                                 & \textrm{est hyperbolique ou quasi-hyperbolique}.
%                           \end{array}
%                   \right\}
%\end{array}
%$$

Le théorème suivant est le point de départ de notre étude. Il caractérise le fait qu'une structure projective est de volume fini en terme d'holonomie. Ce théorème est le resultat principal de l'article \cite{ludo}.

\begin{theo}[M]\label{base}
Une structure projective proprement convexe à bord géodésique principal sur la surface $\Sigma_{g,b,p}$ est de volume fini si et seulement si l'holonomie des lacets qui font le tour d'un bout est parabolique et l'holonomie des lacets élémentaires homotopes à une composante connexe du bord est hyperbolique ou quasi-hyperbolique.
\end{theo}

On rappelle la proposition suivante qui est démontré dans \cite{ludo}.

\begin{prop}[M]\label{nonelem}
Soit $S$ une surface, on munit $S$ d'une structure projective convexe et on note $\Omega$ l'intérieur de la partie proprement convexe donné par la développante de la structure projective de $S$. L'holonomie de tout lacet non élémentaire est un élément hyperbolique qui vérifie que $p^0_{\g} \notin \partial \Omega$.
\end{prop}

On rappelle brièvement la définition d'une décomposition en pantalon.

\begin{defi}
Soit $S$ une surface, une famille de lacets $(c_i)_{i \in I}$ de $S$ est une \emph{décomposition en pantalon de $S$} lorsque les $(c_i)_{i \in I}$ sont des lacets simples, deux à deux disjoints et les composantes connexes de $S$ privée de la réunion des $(c_i)_{i \in I}$ sont des pantalons ouverts.
\end{defi}

\par{
Pour paramétrer l'espace des modules des structures projectives de volume fini sur une surface $S$, nous allons utiliser une décomposition en pantalon de la surface à l'aide d'un ensemble fini de lacets $(c_i)_{i \in I}$. Dans le cas d'une surface de Riemann, on voit apparaître deux types de paramètres: la classe de conjugaison de l'holonomie des lacets $(c_i)_{i \in I}$ et un paramètre de twist le long de chacun des $(c_i)_{i \in I}$. Dans le cas projective convexe, les choses se compliquent un peu. Tout d'abord la classe de conjugaison de l'holonomie des lacets $(c_i)_{i \in I}$ et le paramètre de twist le long de chacun des $(c_i)_{i \in I}$ ne sont plus mesurés à l'aide d'un nombre réel mais à l'aide d'un élément de $\R^2$. Et aussi, alors qu'il existe un et un seul pantalon hyperbolique dont les bords sont de longueur fixée, l'espace des modules des pantalons projectifs proprement convexes dont les classes de conjugaison de l'holonomie des bords est fixée est homéomorphe à $\R^2$.
}
\\
\par{
Lorsque qu'on découpe une surface à bord $S$ en pantalon, on obtient trois types d'holonomie pour les lacets élémentaires des pantalons:

\begin{itemize}
\item Les lacets qui viennent d'un lacet non élémentaire dans $S$,
l'holonomie de ces lacets est hyperbolique.

\item Les lacets qui font le tour d'un bout de $S$, l'holonomie de ces lacets est parabolique.

\item Les lacets homotopes à une composante connexe du bord de $S$, l'holonomie de ces lacets est hyperbolique ou quasi-hyperbolique.
\end{itemize}
}
\begin{defi}
\emph{Un pantalon avec mémoire} est une surface $\Sigma_{0,b,p}$, avec $b+p =3 $, ainsi que la donnée pour chacun des bords d'une orientation et de l'un des mots suivants: $\{$coupure, bord$\}$. On note $P_{i,j,k}^*$ le pantalon avec mémoire qui possède $i$ bords marqués "coupure", $j$ bords marqués "bord" et $k$ pointes. On a une définition naturelle de $\beta_f(P_{i,j,k}^*)$: il s'agit de l'espace des modules des structures projectives convexes marquées de volume fini sur la surface $\Sigma_{0,i+j,k}$ telle que l'holonomie des bords marqués "coupure" est hyperbolique.
\end{defi}

Soient $\Sigma_{g,b,p}$ une surface de caractéristique d'Euler strictement négative et une famille de lacets $(c_i)_{i \in I}$ de $S$ qui définit une décomposition en pantalon de $S$, on note $(P_j)_{j \in J}$ l'ensemble des pantalons obtenus. On se donne une famille $(b_k)_{k=1,...,b}$ de lacets élémentaires homotopes à une composante connexe du bord de $S$ telle que toutes les composantes connexes du bord sont représentées une et une seule fois. On notera $P_j^*$ le pantalon avec mémoire associé à $P_j$ de la façon suivante:
\begin{itemize}
\item Les bords de $P_j$ qui viennent d'un lacet non élémentaire de $S$ sont marqués "coupures".

\item Les bords de $P_j$ qui viennent d'un lacet élémentaire de $S$ sont marqués "bords".
\end{itemize}
L'orientation des bords de $P_j$ est l'orientation donnée par les lacets $(c_i)_{i \in I}$ et $(b_k)_{k=1,...,b}$.

On peut à présent énoncer notre paramétrisation de l'espace des modules. Commençons par énoncer le théorème dans le cas sans bord.

\begin{theo}\label{modulesansbord}
Soient $\Sigma_{g,0,p}$ une surface sans bord de caractéristique d'Euler strictement négative et une famille de lacets $(c_i)_{i \in I}$ de $S$ qui définit une décomposition en pantalon de $S$, on note $(P_j^*)_{j \in J}$ l'ensemble des pantalons avec mémoire obtenus. Alors,
\begin{enumerate}
\item Les lacets $(c_i)_{i \in I}$ ne sont pas élémentaires.

\item $|I| = 3g-3+p$ et $|J|=-\chi(\Sigma_{g,b,p})=2g-2+p$.

\item L'holonomie des $c_i$ pour $i \in I$ est hyperbolique.

\item L'application naturelle $\beta_f(\Sigma_{g,0,p}) \rightarrow
\underset{j \in J}{\prod} \beta_f(P_j^*)$ est une fibration qui admet
une action simplement transitive préservant les fibres du groupe
$(\R^2)^I$.

\item On note $(c_j^r)_{r=1...l_j}$ la sous-famille de $(c_i)_{i \in I}$ des $l_j$ lacets élémentaires du pantalon $P_j$ homotopes à une composante connexe du bord du pantalon $P_j$. L'application:
$$
\begin{array}{ccc}
\beta_f(P_j^*) & \rightarrow & \mathfrak{R}^{l_j}\\
\mathcal{P}             & \mapsto     & (\lambda(Hol(c_j^r)),\tau(Hol(c_j^r)))_{r=1...l_j}
\end{array}
$$
est une fibration dont les fibres sont homéomorphes à $\R^2$.

\item En particulier, $\beta_f(\Sigma_{g,0,p})$ est homéomorphe à $\R^{16g-16+6p}$.

\end{enumerate}
\end{theo}

%Pour décrire l'espace des modules dans le cas à bord, nous aurons besoin de la définition suivante.

%\begin{defi}
%Un \emph{espace topologique stratifié} de dimension $n$ est un espace topologique $X$ muni d'une suite croissante $X_0 \subset ... \subset X_n = X$ de fermé de $X$ tel que le sous-espace $X_i \setminus X_{i-1}$ est une variété topologique de dimension $i$ ou l'ensemble vide, pour tout $i=1,...,n-1$. Les composantes connexes de $X_i \setminus X_{i-1}$ sont appellées les \emph{strates} de dimension $i$ de $X$.
%\end{defi}

On donne à présent l'énoncé dans le cas à bord.

\begin{theo}\label{module}
Soient $\Sigma_{g,b,p}$ une surface à bord de caractéristique d'Euler strictement négative et une famille de lacets $(c_i)_{i \in I}$ de $S$ qui définit une décomposition en pantalon de $S$, on note $(b_k)_{k=1,...,b}$ une famille de lacets élémentaires homotopes à une composante connexe du bord de $S$ telle que toutes les composantes connexes du bord sont représentées une et une seule fois. On note $(P_j^*)_{j \in J}$ l'ensemble des pantalons avec mémoire obtenus. Alors,
\begin{enumerate}
\item Les lacets $(c_i)_{i \in I}$ ne sont pas élémentaires.

\item $|I| = 3g-3+p+b$ et $|J|=-\chi(\Sigma_{g,b,p})=2g-2+p+b$.

\item L'holonomie des $c_i$ pour $i \in I$ est hyperbolique.

\item L'application naturelle $\beta_f(\Sigma_{g,b,p}) \rightarrow
\underset{j \in J}{\prod} \beta_f(P_j^*)$ est une fibration qui admet
une action simplement transitive préservant les fibres du groupe
$(\R^2)^I$.

\item On note $(c_j^r)_{r=1...l_j}$ la sous-famille de $(c_i)_{i \in I}$ des $l_j$ lacets élémentaires du pantalon $P_j$ qui sont homotopes à une composante connexe du bord du pantalon $P_j$ marqué "coupure". On note $(b_j^s)_{s=1...m_j}$ la sous-famille de $(b_k)_{k =1,...,b}$ des $m_j$ lacets élémentaires du pantalon $P_j$ qui sont homotopes à une composante connexe du bord du pantalon $P_j$ marqué "bord". L'application:
$$
\begin{array}{ccc}
\beta_f(P_j^*) & \rightarrow & \mathfrak{R}^{l_j} \times \widehat{\mathfrak{R}}^{m_j}\\
\mathcal{P}              & \mapsto     & (\lambda(Hol(c_j^r)),\tau(Hol(c_j^r)))_{r=1...l_j} \times (\lambda(Hol(b_j^s)),\tau(Hol(b_j^s)))_{s=1...m_j}
\end{array}
$$
est une fibration dont les fibres sont homéomorphes à $\R^2$.

\item En particulier, $\beta_f(\Sigma_{g,b,p})$ est une variété topologique à bord de dimension $16g-16+6p+8b$ homéomorphe à $\R^{16g-16+6p+7b} \times [0,1]^b$.
\end{enumerate}
\end{theo}

La démonstration de ce théorème va nécessiter plusieurs lemmes. Mais commençons plutôt par donner les étapes de la démonstration. Les deux premiers points sont des résultats classiques de topologie des surfaces. Le troisième point est une conséquence du premier point et de la proposition \ref{nonelem}. Le quatrième point est un résultat de Goldman (\cite{Gold1}) dont nous donnons une démonstration au paragraphe \ref{point4} pour la commodité du lecteur. Le cinquième point est le point le plus technique, l'idée vient de Goldman (\cite{Gold1}) mais cette fois-ci, nous aurons besoin d'un réel travail pour pouvoir utiliser ces idées, on verra cela au paragraphe \ref{point5}. Le sixième et dernier point est une simple conséquence des points 4. et 5. et du calcul suivant.

\begin{itemize}
\item L'holonomie des $(c_i)_{i \in I}$  donne $|I|$ éléments de $\mathfrak{R}$, soit $2(3g-3+p)$ paramètres.

\item Dans le cas avec bord, les bords  fournissent $b$ éléments de $\widehat{\mathfrak{R}}$.

\item Les pantalons $(P_j)_{j \in J}$  donnent deux paramètres chacun, soit $2(2g-2+p)$ paramètres.

\item Le recollement de deux pantalons le long d'un $c_i$  fournit deux paramètres, soit $2(3g-3+p)$ paramètres.

\item Bilan: l'espace des modules est homéomorphe à $\mathfrak{R}^I \times \widehat{\mathfrak{R}}^b \times (\R^2)^J \times (\R^2)^I$ qui est une variété topologique à bord de dimension $2(3g-3+p)+ 2b + 2(2g-2+p) + 2(3g-3+p)= 16g-16+8b+6p$.
\end{itemize}

\begin{rem}
Décrivons un peu la structure de $\beta_f(\Sigma_{g,b,p})$. L'intérieur de $\beta_f(\Sigma_{g,b,p})$ est homéomorphe à $\R^{16g-16+8b+6p}$. Il s'agit des structures pour lesquels l'holonomie des lacets homotopes à une composante connexe du bord est hyperbolique.

Plus généralement, on peut décomposer $\beta_f(\Sigma_{g,b,p})$ en une réunion de variétés lisses (strates) toutes homéomorphes à des boules. Et, si $d \leqslant b$ l'espace $\beta_f(\Sigma_{g,b,p})$ possède $2^d \times C^d_{b}$ strates de codimension $d$. En effet, la réunion des strates de codimension $d$ est l'ensemble des structures dont l'holonomie de $d$ lacets homotopes à une composante connexe du bord est quasi-hyperbolique. Choisir une strate de dimension $d$, c'est donc choisir $d$ lacets parmi les lacets $(b_k)_{k=1,...,b}$ et ensuite pour chacun d'eux choisir si le point fixe répulsif de l'holonomie est $C^1$ ou non. Il n'y a pas de strate de codimension $d$, si $d > b$. Enfin, on peut remarquer que l'adhérence des strates de codimension $d$ est l'ensemble des strates de codimension $d' \geqslant d$.
\end{rem}

\subsection{Démonstration du quatrième point du théorème \ref{module}}\label{point4}

Le but de ce paragraphe est de montrer le quatrième point du théorème \ref{module}. Commençons par rappeler quelques propositions classiques de géométrie hyperbolique qui sont encore vrai dans pour les surfaces projectives proprement convexes.

\begin{prop}\label{lacethomoo}
Soit $S$ une surface, on munit $S$ d'une structure projective proprement convexe. Alors,
tout lacet simple non élémentaire tracé sur $S$ est librement homotope à une géodésique simple.
\end{prop}

\begin{prop}\label{lacethomoo2}
Soit $S$ une surface, on munit $S$ d'une structure projective proprement convexe. Soient $c_1$ et $c_2$ deux lacets simples et non élémentaires tracés sur $S$, on note $\lambda_1$ (resp. $\lambda_2$) l'unique géodésique librement homotope à $c_1$ (resp. $c_2$). Si $c_1$ est simple alors $\lambda_1$ est simple. Si les lacets $c_1$ et $c_2$ ne s'intersectent pas, alors les géodésiques $\lambda_1$ et $\lambda_2$ ne s'intersectent pas.
\end{prop}

\par{
Les démonstrations sont identiques à celle du monde hyperbolique. On peut les trouver dans \cite{Gold1,ludo}.
}
\\

\par{
On considère une surface $S$ et $c$ un lacet simple et non élémentaire tracé sur $S$. On note $S|c$ la surface topologique obtenue en retirant $c$ à $S$ et en ajoutant deux bords correspondants à $c$. C'est possible car $S$ est orientable. On note $\beta_f(S|c^*)$ l'espace des modules des structures projectives de volume fini sur la surface $S|c$ tel que les holonomies des deux bords correspondants à $c$ sont hyperboliques et conjuguées.
}
\\
\par{
Il y a une application naturelle de $\varphi_{|c}: \beta_f(S) \rightarrow \beta_f(S|c^*)$. Expliquons sa construction de façon succinte. On considère une surface projective proprement convexe $\mathcal{S} \in \beta_f(S)$, elle est donnée par un homéomorphisme de $S$ vers le quotient d'une partie proprement convexe $\C$ de $\P$ par un sous-groupe discret $\G$ de $\s$ qui préserve $\C$. Le lacet $c$ est librement homotope à une géodésique de $\mathcal{S}$, on peut donc supposer que tout relevé de $c$ à $\C$ est un segment ouvert de $\C$ dont les extrémités sont sur le bord de $\C$. On choisit un relevé $\lambda$ de $c$, comme le lacet $c$ est simple, les images de $\lambda$ par $\G$ ne s'intersectent pas.}
\\
\par{
Pour construire $\varphi_{|c}$, il faut distinguer deux cas: la surface $S|c$ est connexe ou bien la surface $S|c$ possède deux composantes connexes. Dans le premier cas, on note $\C'_{|c}$ une composante connexe de $\C - \underset{\g \in \G}{\bigcup}\g \lambda$. On note $\G_{|c}$ le stabilisateur dans $\G$ de $\C'_{|c}$. On note $\C_{|c}$ le convexe obtenu en ajoutant à $\C'_{|c}$ les relevés de $c$ qui bordent $\C'_{|c}$. Le couple $(\C_{|c},\G_{|c})$ fournit la structure voulue sur $S|c$, celle-ci est clairement proprement convexe et elle est de volume fini d'après le théorème \ref{base}.
}
\\
\par{
Dans le second cas, on note $\C'^1_{|c}$ et $\C'^2_{|c}$ les composantes connexes de $\C - \underset{\g \in \G}{\bigcup} \g \lambda$ qui bordent $\lambda$. On note $\G^1_{|c}$ (resp. $\G^2_{|c}$) le stabilisateur dans $\G$ de $\C'^1_{|c}$ (resp. $\C'^2_{|c}$). On note $\C^1_{|c}$ (resp. $\C^2_{|c}$) le convexe obtenu en ajoutant à $\C'^1_{|c}$ (resp. $\C'^2_{|c}$) les relevés de $c$ qui bordent $\C'^1_{|c}$ (resp. $\C'^2_{|c}$). Les couples $(\C^1_{|c},\G^1_{|c})$ et $(\C^2_{|c},\G^2_{|c})$  fournissent la structure voulue sur $S|c$, celle-ci est clairement proprement convexe et elle est de volume fini d'après le théorème \ref{base}.
}
\\
\par{
On peut à présent énoncer la proposition \ref{collag} dont le point 4. est une conséquence évidente.
}
\begin{prop}\label{collag}
Soient $S$ une surface et $c$ un lacet simple et non élémentaire de $S$, l'application naturelle de $\varphi_{|c}:\beta_f(S) \rightarrow \beta_f(S|c^*)$ est une fibration qui admet une action simplement transitive du groupe $\R^2$ sur ces fibres.
\end{prop}

Dans le paragraphe \ref{constr1} nous allons construire cette action et dans le paragraphe \ref{simpltrans} nous montrerons que cette action est simplement transitive.

\subsubsection{Construction de l'action de $\R^2$}\label{constr1}

%Soient $S$ une surface et un lacet $c$ simple et non élémentaire tracé sur $S$, on considère une surface projective proprement convexe $\mathcal{S} \in \beta_f(S)$, elle est donnée par un homéomorphisme de $S$ vers le quotient d'une partie proprement convexe $\C$ de $\P$ par un sous-groupe discret $\G$ de $\s$ qui préserve $\C$. Le lacet $c$ est librement homotope à une géodésique de $\mathcal{S}$, on peut donc supposer que tout relevé de $c$ à $\C$ est un segment ouvert de $\C$ dont les extrémités sont sur le bord de $\C$. On choisit un relevé $\lambda$ de $c$, comme le lacet $c$ est simple, les images de $\lambda$ par $\G$ ne s'intersectent pas.

On reprend les notations du paragraphe précédent. On sait que l'holonomie $\g$ de $c$ est hyperbolique avec $p^0_{\g} \notin \C$. On peut supposer que l'élément $\g$ est donné par une matrice diagonale à diagonale strictement positive et que les entrées de la diagonale de $\g$ sont ordonnées par ordre croissant.

Soient $(u,v) \in \R^2$, on définit les deux matrices suivantes:

$$
\begin{array}{cccc}
T^u= &
\left(
\begin{array}{ccc}
e^{-u} & 0 & 0\\
0    & 1 & 0\\
0    & 0 & e^u\\
\end{array}
\right)
&
U^v= &
\left(
\begin{array}{ccc}
e^{-v} & 0      & 0\\
0    & e^{2v} & 0\\
0    & 0      & e^{-v}\\
\end{array}
\right)
\end{array}
$$

\par{
La composante neutre du centralisateur de $\g$ dans $\s$ est le groupe $D$ isomomorphe à $\R^2$ donné par $D=\{ T^u U^v \,|\, (u,v) \in \R^2 \}$. Nous allons définir une action de $D$ sur $\beta_f(S)$ qui préserve la fibration $\varphi_{|c}: \beta_f(S) \rightarrow \beta_f(S|c^*)$.
}
\\
\par{
On se donne deux réels $u$ et $v$ et on va construire une nouvelle structure projective proprement convexe sur $S$. On note $dev_0$ la développante de $\mathcal{S}$ et $\rho_0$ son holonomie. Il faut distinguer deux cas: la surface $S|c$ est connexe ou bien la surface $S|c$ possède deux composantes connexes.
}
\\
\par{
On commence par le second cas car il est plus facile à visualiser. On note $\C'_1$ et $\C'_2$ les deux composantes connexes de $\C - \underset{\g \in \G}{\bigcup} \g \lambda$ qui bordent $\lambda$. On note $\G_1$ (resp. $\G_2$) le stabilisateur dans $\G$ de $\C'_1$ (resp. $\C'_2$). Le groupe $\G$ est le produit amalgamé du groupe  $\G_1$ et $\G_2$ le long du groupe engendré par $\g$. La nouvelle holonomie $\rho_{u,v}$ est définie de la façon suivante:
}

\begin{itemize}
\item Si $\delta \in \G_1$ alors on pose $\rho_{u,v}(\delta)= \rho_0(\delta)$.

\item Si $\delta \in \G_2$ alors on pose $\rho_{u,v}(\delta)= T^u U^v \rho_0(\delta) U^{-v} T^{-u}$.

\item Cette définition n'est pas ambiguë car $T^u U^v$ commute avec $\g$.
\end{itemize}

La nouvelle développante est l'unique homéomorphisme $\rho_{u,v}$-équivariant qui prolonge l'application $dev_{u,v}:\C'_1 \cup \C'_2 \rightarrow \P$ suivante:

\begin{itemize}
\item Si $x \in \C'_1$ alors on pose $dev_{u,v}(x)= dev_0(x)$.

\item Si $x \in \C'_2$ alors on pose $dev_{u,v}(x)= T^u U^v dev_0(x)$.

\item L'existence et l'unicité du prolongement de $dev_{u,v}$ à $\C=\widetilde{S}$ est évidente.
\end{itemize}
\vspace{.5cm}
\par{
Dans le premier cas, on note $\C'$ une composante connexe de $\C - \underset{\g \in \G}{\bigcup}\g \lambda$. On note $\G'$ le stabilisateur dans $\G$ de $\C'$. Enfin, on note $\alpha$ un élément de $\G$ qui correspond à un chemin simple tracé sur $S|c$ reliant les deux bords de $S|c$ correspondant à $c$. Nous allons modifier la développante $dev_0$ de $S$ ainsi que son holonomie $\rho_0$. Le groupe $\G$ est l'HNN-extension de $\G'$ relativement à $\alpha$. La nouvelle holonomie $\rho_{u,v}$ est définie de la façon suivante:
}

\begin{itemize}
\item Si $\delta \in \G'$ alors on pose $\rho_{u,v}(\delta)= \rho_0(\delta)$.

\item On pose $\rho_{u,v}(\alpha)= T^u U^v \rho_0(\alpha)$.

\item Cette définition ne dépend pas du choix de l'élément $\alpha$ car $T^u U^v$ commute avec l'élément $\g$.

     %En effet, si $\alpha$ et $\alpha'$ désignent deux éléments de $\G$ qui correspondent à un chemin simple tracé sur $S|c$ reliant les deux bords de $S|c$ correspondant à $c$. Alors, il existe des entiers $p,q \in \Z$ tel que $\alpha' = c^n \alpha c^p$. La définition de $\rho_{u,v}$ ne dépend donc pas du choix du lacet $\alpha$.
\end{itemize}
\par{
La nouvelle développante est l'unique homéomorphisme $\rho_{u,v}$-équivariant qui prolonge l'application $dev_{u,v}|_{\C'}= dev_0|_{\C'}$.
}
\\
\par{
On a ainsi défini pour tout lacet simple non élémentaire $c$ tracé sur la surface projective proprement convexe $\mathcal{S} \in \beta_f(S)$ une action du centralisateur $D$ de $\Hol(c)$ sur $\mathcal{S}$.
}
\\
\par{
Comme $D$ est isomorphe à $\R^2$, nous venons de construire pour tout lacet simple non élémentaire $c$ tracé sur la surface topologique $S$ une action de $\R^2$ sur $\beta_f(S)$ qui préserve la fibration $\varphi_{|c}: \beta_f(S) \rightarrow \beta_f(S|c^*)$. Ainsi, pour tout couple $(u,v) \in \R^2$, tout lacet simple tracé sur $S$ et pour tout surface projective proprement convexe $\mathcal{S} \in \beta_f(S)$, on notera $\mathcal{S}^{(u,v)}$ la surface projective proprement convexe de volume fini obtenue par l'action de $(u,v)$ sur $\mathcal{S}$ que l'on vient de définir. Cette action de $\R^2$ sur $\beta_f(S)$ est libre. En effet, il est clair que $\rho_{u,v} \neq \rho_0$ pour tout couple $(u,v) \in \R^2$ (Il suffit de distinguer les cas $S|c$ connexe et $S|c$ non connexe).
}

\subsubsection{L'action de $\R^2$ est simplement transitive}\label{simpltrans}

Nous allons montrer le lemme suivant qui conclut la démonstration de la proposition \ref{collag}.

\begin{lemm}
Soient $S$ une surface, $c$ un lacet simple non élémentaire de $S$ et $\mathcal{S}$ et $\mathcal{T}$ deux surfaces projectives proprement convexes de $\beta_f(S)$ telles que $\varphi_{|c}(\mathcal{S})=\varphi_{|c}(\mathcal{T})$, alors il existe un unique couple $(u,v) \in \R^2$ tel que $\mathcal{T}= \mathcal{S}^{(u,v)}$.
\end{lemm}

\begin{proof}
Comme l'action de $\R^2$ sur $\beta_f(S)$ est libre, il n'y a que l'existence à montrer. On ne fait que le cas où $S|c$ possède deux composantes connexes, l'autre cas est analogue. On note $\C^1_{\mathcal{S}}$ et $\C^2_{\mathcal{S}}$ (resp. $\C^1_{\mathcal{T}}$ et $\C^2_{\mathcal{T}}$) les parties proprement convexes associées aux deux composantes connexes de $\mathcal{S}$ (resp. $\mathcal{T}$) données par la développante de sa structure projective. On suppose bien sûr que $\C^1_{\mathcal{S}}$ et $\C^1_{\mathcal{T}}$ (resp. $\C^2_{\mathcal{S}}$ et $\C^2_{\mathcal{T}}$) correspondent à la même composante connexe de $S|c$.

L'hypothèse $\varphi_{|c}(\mathcal{S})=\varphi_{|c}(\mathcal{T})$ entraine que les parties proprement convexes $\C^1_{\mathcal{S}}$ et $\C^1_{\mathcal{T}}$ (resp. $\C^2_{\mathcal{S}}$ et $\C^2_{\mathcal{T}}$) sont projectivement équivalentes.

On peut supposer que l'holonomie $\g$ de $c$ pour ces 4 structures projectives est donnée par une matrice diagonale à diagonale strictement positive, et que les entrées de la diagonale de $\g$ sont rangées par ordre croissant. Les points $p^+_{\g},p^-_{\g},p^0_{\g}$ sont donc les mêmes pour ces 4 structures, ils définissent un pavage de $\P$ par 4 triangles fermés. Deux de ces triangles bordent l'axe de $\g$, on les note $\Delta_1$ et $\Delta_2$. La dynamique des éléments hyperboliques montre que, quitte à renuméroter, on peut supposer que $\C^1_{\mathcal{S}}$ et $\C^1_{\mathcal{T}}$ (resp. $\C^2_{\mathcal{S}}$ et $\C^2_{\mathcal{T}}$) sont inclus dans $\Delta_1$ (resp. $\Delta_2$).

Les parties proprement convexes $\C^1_{\mathcal{S}}$ et $\C^1_{\mathcal{T}}$ sont projectivement équivalentes. On peut donc supposer que $\C^1_{\mathcal{S}} = \C^1_{\mathcal{T}}$. De même, les parties proprement convexes $\C^2_{\mathcal{S}}$ et $\C^2_{\mathcal{T}}$ sont projectivement équivalentes. Par conséquent, il existe une transformation projective $f$ qui identifie $\C^2_{\mathcal{S}}$ sur $\C^2_{\mathcal{T}}$. On peut supposer que celle-ci préserve l'axe de $\g$. Par conséquent, elle doit aussi préserver le point $p^0_{\g}$ puisque c'est l'intersection des demi-tangentes (différentes de l'axe de $\g$) à $\C^2_{\mathcal{S}}$ et $\C^2_{\mathcal{T}}$ aux points $p^+_{\g}$ et $p^-_{\g}$. Par conséquent, la transformation projective $f$ fixe les points $p^+_{\g}$, $p^-_{\g}$ et $p^0_{\g}$, c'est donc un élément du centralisateur de $\g$. Par suite, l'élément $f$ appartient au groupe $D$, et la structure $\mathcal{T}$ sur la surface topologique $S$ est obtenue par l'action de $f$ sur $\mathcal{S}$. C'est ce qu'il fallait montrer.
\end{proof}

\subsection{Démonstration du cinquième point du théorème \ref{module}}\label{point5}

\subsubsection{Les pantalons et l'objet combinatoire}
\par{
Pour éviter les confusions sur l'objet désigné par le mot triangle, on utilisera la convention suivante: un \emph{triangle fermé} est un fermé d'une carte affine qui est l'intersection de trois demi-plans fermés en position générique. Un \emph{triangle ouvert} est l'intérieur (qui est toujours non vide) d'un triangle fermé. Un \emph{triangle épointé} est un triangle fermé sans ses sommets. Un \emph{triangle topologique} est l'image d'un triangle épointé par un homéomorphisme.
}
\\
\par{
L'idée de Goldman est la suivante: nous allons associer à tout pantalon projectif proprement convexe $\mathcal{P}$ un objet combinatoire, à savoir quatre triangles dont la réunion est un hexagone, muni de trois applications projectives qui identifient certains triangles entre eux (Voir figure \ref{hexa}). Le groupe $\G$ engendré par ces trois applications est isomorphe au groupe libre à 2 générateurs. La réunion des quatre triangles épointés et de leurs images sous $\G$ forment un ouvert $\Omega$ proprement convexe. Le quotient $\Quo$ est le pantalon $\mathcal{P}$.
}

\begin{figure}[!h]
\begin{center}
\includegraphics[trim=-7cm 14cm 0cm 0cm, clip=true, width=12cm]{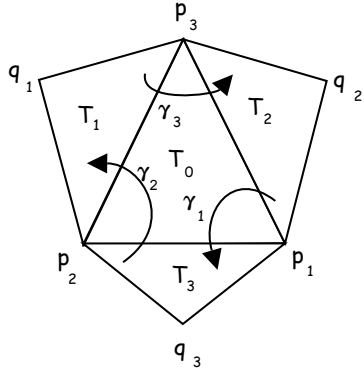}
%\centerline{\psfig{figure=hexa.ps, width=6cm}}
\caption{$L'hexagone$} \label{hexa}
\end{center}
\end{figure}

\par{
La démonstration se fait en trois parties. On commence par construire à partir d'un pantalon projectif proprement convexe l'objet combinatoire (paragraphe \ref{constr}). Ensuite, on montre que, si l'on part de l'objet combinatoire, alors on obtient bien un pantalon projectif proprement convexe (paragraphe \ref{convexe}). Enfin, on calcule l'espace des modules de cet objet combinatoire (paragraphe \ref{computation}).
}
\\
\par{
On souhaite comprendre l'espace des modules des pantalons projectifs proprement convexes dont les classes de conjugaison des holonomies des lacets élémentaires sont fixées. On introduit donc l'espace suivant, on se donne $\delta_1,\delta_2,\delta_3 \in \s$ et on pose:
}

$$
\Q'_{\delta_1,\delta_2,\delta_3} =
\left\{
\begin{array}{l}
(T_0,T_1,T_2,T_3,\g_1,\g_2,\g_3) \textrm{ tel que:}\\
\\
1.\, T_0,T_1,T_2,T_3 \textrm{ sont des triangles épointés et le triangle } T_0\\
\textrm{intersecte } T_1,T_2,T_3 \textrm{ le long de ses 3 arêtes}\\
\\
2.\, \underset{i=0,...,3}{\bigcup} \overline{T_i} \textrm{ est un hexagone }\\
\textrm{ convexe qui possède exactement 6 côtés}\\
\\
3.\, \g_3\g_2\g_1=1\\
\\
4.\, \g_1(T_2) = T_3, \, \g_2(T_3) = T_1, \, \g_3(T_1) = T_2\\
\\
5.\, \textrm{L'élément } \g_i \textrm{ est conjugué à } \delta_i\textrm{ dans } \s, \textrm{ pour } i=1,...,3.\\
\end{array}
\right\}
$$
\par{
L'action naturelle de $\s$ sur $\Q'_{\delta_1,\delta_2,\delta_3}$ est libre et propre: on s'intéresse à l'espace quotient $\Q_{\delta_1,\delta_2,\delta_3}=\Q'_{\delta_1,\delta_2,\delta_3}/_{\s}$.
}
\\
\par{
Soient $\delta_1,\delta_2,\delta_3$ des éléments de $\s$ qui sont hyperboliques, quasi-hyperboliques ou paraboliques. Soit $P$ la surface $\Sigma_{0,3,0}$, on numérote les composantes connexes du bord de $P$ de 1 à 3. On choisit aussi une orientation de chacun des bords. On note $P_{\delta_1,\delta_2,\delta_3}$ le pantalon à bord obtenue à partir de $P$ en retirant le bord $i$ lorsque l'élément $\delta_i$ est parabolique. On notera $\beta_f(P_{\delta_1,\delta_2,\delta_3})$ l'espace des modules des structures projectives proprement convexes sur le pantalon $P_{\delta_1,\delta_2,\delta_3}$ tel que l'holonomie du lacet élémentaire orienté numéroté $i$ est conjuguée à $\delta_i$, pour $i=1,...,3$.
}
\\
\par{
Les deux prochains paragraphes sont consacrés à la démonstration de cette proposition.
}
\begin{prop}\label{panta-combi}
Soient $\delta_1,\delta_2,\delta_3$ trois éléments hyperboliques, quasi-hyperboliques ou paraboliques de $\s$, l'espace $\beta_f(P_{\delta_1,\delta_2,\delta_3})$ est homéomorphe à $\Q_{\delta_1,\delta_2,\delta_3}$.
\end{prop}

\subsubsection{Construction de l'objet combinatoire}\label{constr}

Soit $\mathcal{P} \in \beta_f(P_{\delta_1,\delta_2,\delta_3})$, nous allons découper l'intérieur de $\mathcal{P}$ à l'aide de deux triangles idéaux épointés. Pour cela, on munit le pantalon à bord $P_{\delta_1,\delta_2,\delta_3}$ d'une structure hyperbolique de volume fini. L'holonomie du lacet correspondant au bord $i$ est donc hyperbolique si et seulement si $\delta_i$ est hyperbolique ou quasi-hyperbolique. Sinon, elle est parabolique.

On triangule le pantalon $P_{\delta_1,\delta_2,\delta_3}$ à l'aide de 3 géodésiques $l_i$ pour la structure hyperbolique que l'on a choisi. Si les éléments $\delta_{i+1}$ et $\delta_{i+2}$ sont hyperboliques ou quasi-hyperboliques alors la géodésique $l_i$ est l'unique géodésique qui s'accummule sur les bords $i+1 \, mod \, 3 $ et $i+2 \, mod \, 3 $. Si les éléments $\delta_{i+1}$ et $\delta_{i+2}$ sont paraboliques alors la géodésique $l_i(t)$ est l'unique géodésique qui est incluse dans tout voisinage du bout $i+1 \, mod \, 3 $ (lorsque $t \rightarrow -\infty$ ) et dans tout voisinage du bout $i+2 \, mod \, 3 $ (lorsque $t \rightarrow +\infty$). Si l'élément $\delta_{i+1}$ est parabolique et l'élément $\delta_{i+2}$ est hyperbolique ou quasi-hyperbolique alors $l_i(t)$ est l'unique géodésique qui est incluse dans tout voisinage du bout $i+1 \, mod \, 3 $ (lorsque $t \rightarrow -\infty$ ) et s'accumule le bord $i+2 \, mod \, 3 $. Le dernier cas se traite de la même façon.

Les dessins de la figure \ref{pantalon} résument la construction de la triangulation de $P_{\delta_1,\delta_2,\delta_3}$. La figure \ref{haut} résume la construction de l'objet combinatoire.

\begin{figure}[!h]
\centerline{\psfig{figure=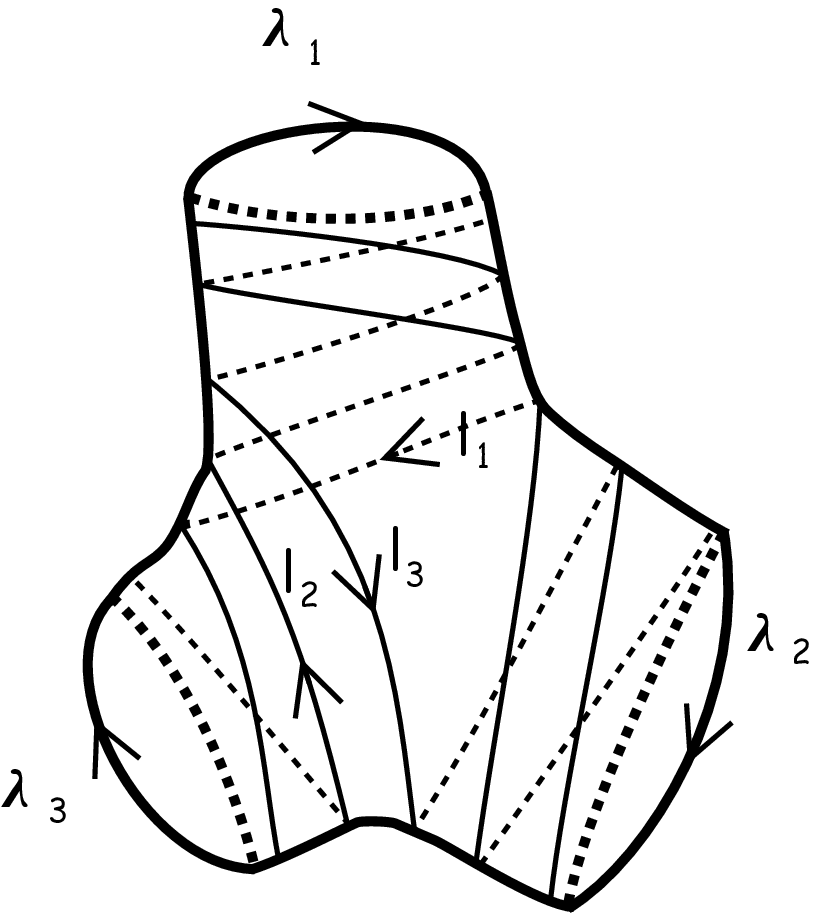, trim=-6cm 13cm 5cm 0cm, width=10cm}
 \hspace{3em}
\psfig{figure=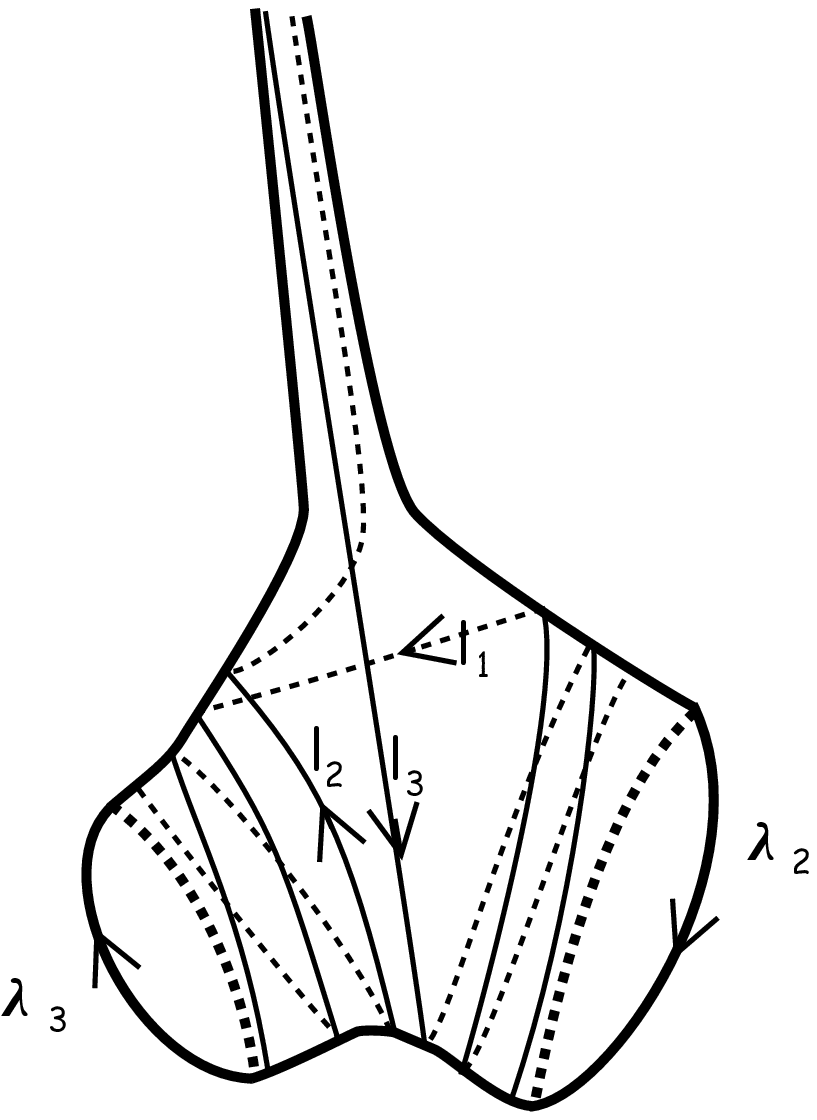, trim=0cm 13cm 0cm 0cm, width=9cm}}
 \vspace{3em}
\centerline{\psfig{figure=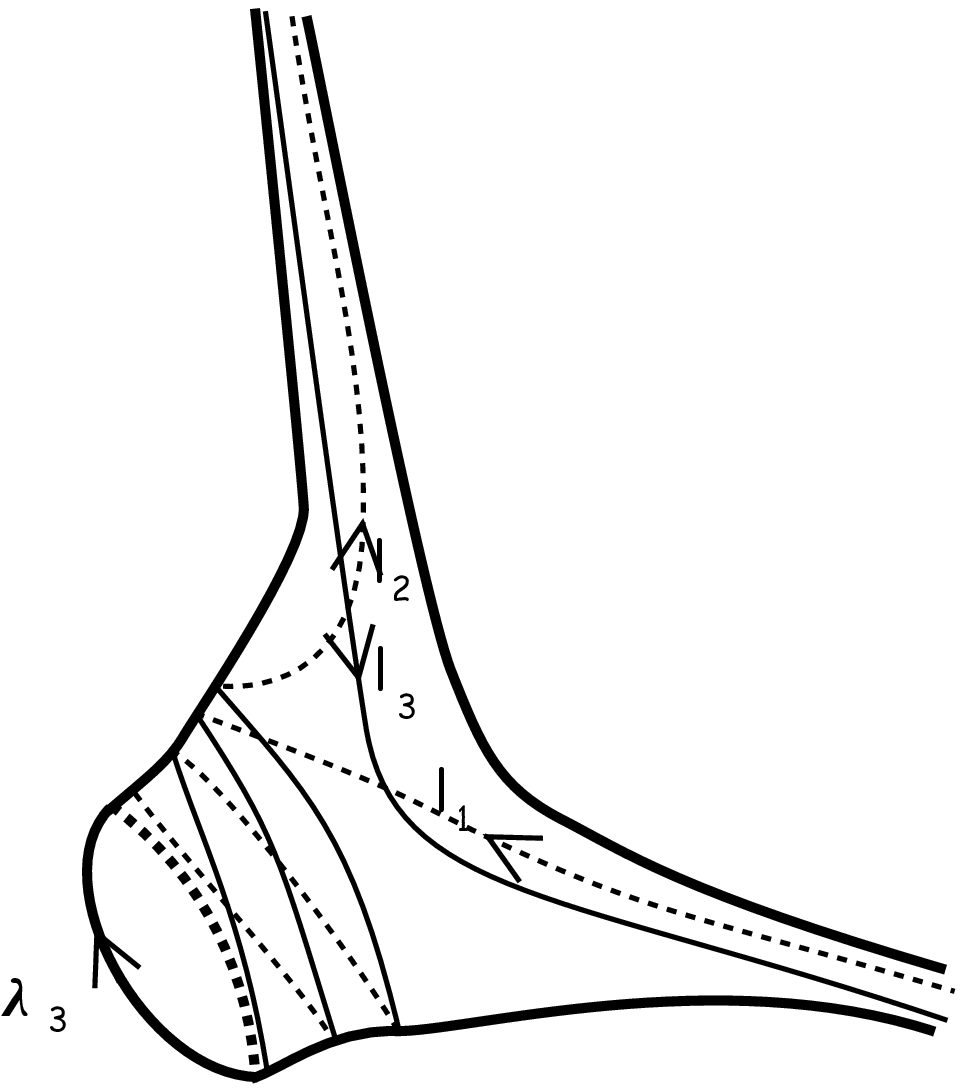, trim=-6cm 13cm 5cm 0cm, width=9cm}
 \hspace{3em}
\psfig{figure=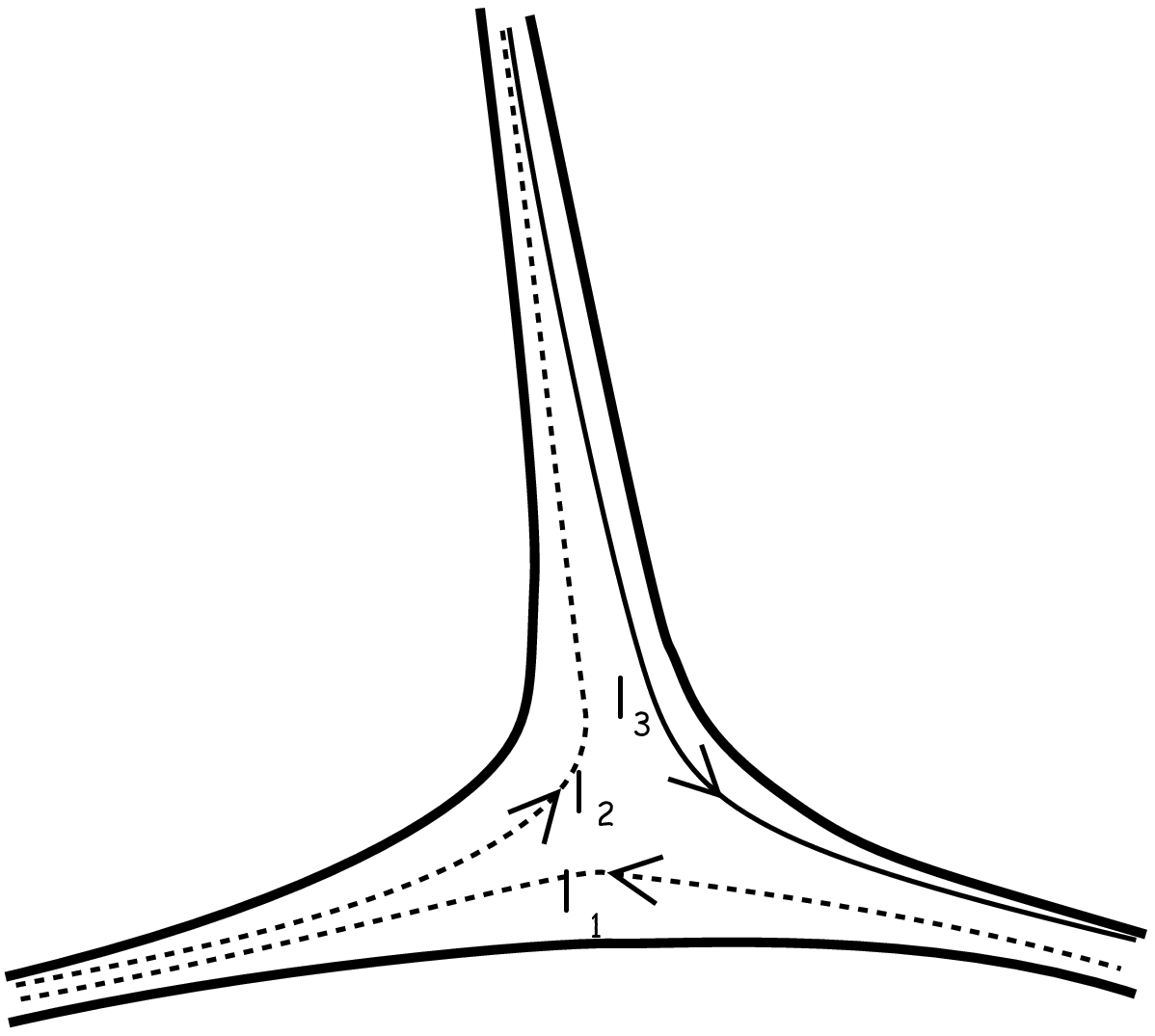, trim=0cm 13cm 0cm 0cm, width=9cm}
}
\caption{Pantalons avec zéro, une, deux et enfin trois pointes} \label{pantalon}
\end{figure}

\begin{figure}[!h]
\centerline{\psfig{figure=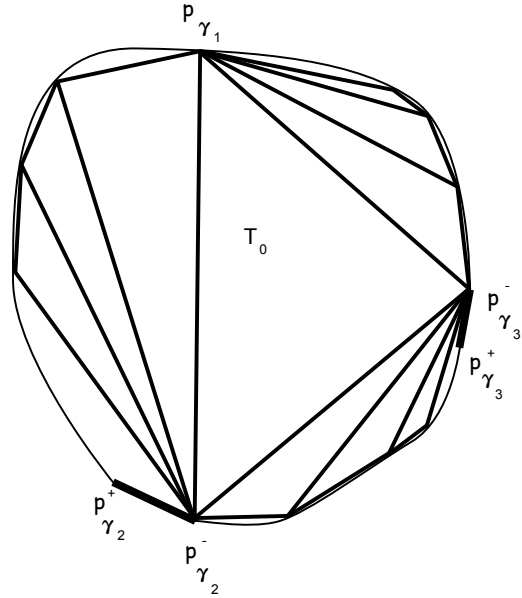, trim=-4cm 9cm 0cm 0cm, width=10cm}}
\caption{Revêtement universel d'un pantalon projectif proprement convexe} \label{haut}
\end{figure}

La forme des chemins $l_1,l_2,l_3$ dépend de l'holonomie du bout. Nous allons montrer que l'on peut supposer que ces chemins sont des géodésiques simples et qu'elles ne s'intersectent pas.

Commençons par analyser ces dessins d'un point de vue topologique, on note $S$ l'intérieur de $P_{\delta_1,\delta_2,\delta_3}$. Les chemins $l_1$, $l_2$ et $l_3$ se relèvent en des chemins simples $\widetilde{l_1}$, $\widetilde{l_2}$ et $\widetilde{l_3}$ tracés sur le revêtement universel $\widetilde{S}$ de $S$. Chacun des chemins $\widetilde{l_i}$ vérifie que $\widetilde{S}-\widetilde{l_i}$ possède deux composantes connexes. Ces chemins et leurs images par le groupe fondamental $\G$ de $S$ définissent une triangulation de $\widetilde{S}$. Deux triangles topologiques adjacents définissent un domaine fondamental pour l'action de $\G$ sur $\widetilde{S}$. La triangulation ainsi obtenue de $\widetilde{S}$ est la triangulation de Faray.

Nous allons montrer que chacun de ces chemins peut être supposés géodésique. Pour simplifier la discussion, on ne fait que le cas où $\delta_1$ est parabolique et $\delta_2$, $\delta_3$ hyperbolique. Les autres cas se traitent de façon analogue. On note $\C$ le revêtement universel de $\mathcal{P}$ donné par la développante de $\mathcal{P}$. L'ensemble $\C$ est un convexe proprement convexe de $\P$. On note $\g_1$ (resp. $\g_2$ resp. $\g_3$) les représentants de l'holonomie des lacets $\lambda_1$ (resp. $\lambda_2$ resp. $\lambda_3$) donnés par le choix du point base $x_0$ sur $\mathcal{P}$ comme sur le dessin \ref{lacet}. Ils vérifient la relation $\g_3\g_2\g_1=1$.

\begin{figure}[!h]
\centerline{\psfig{figure=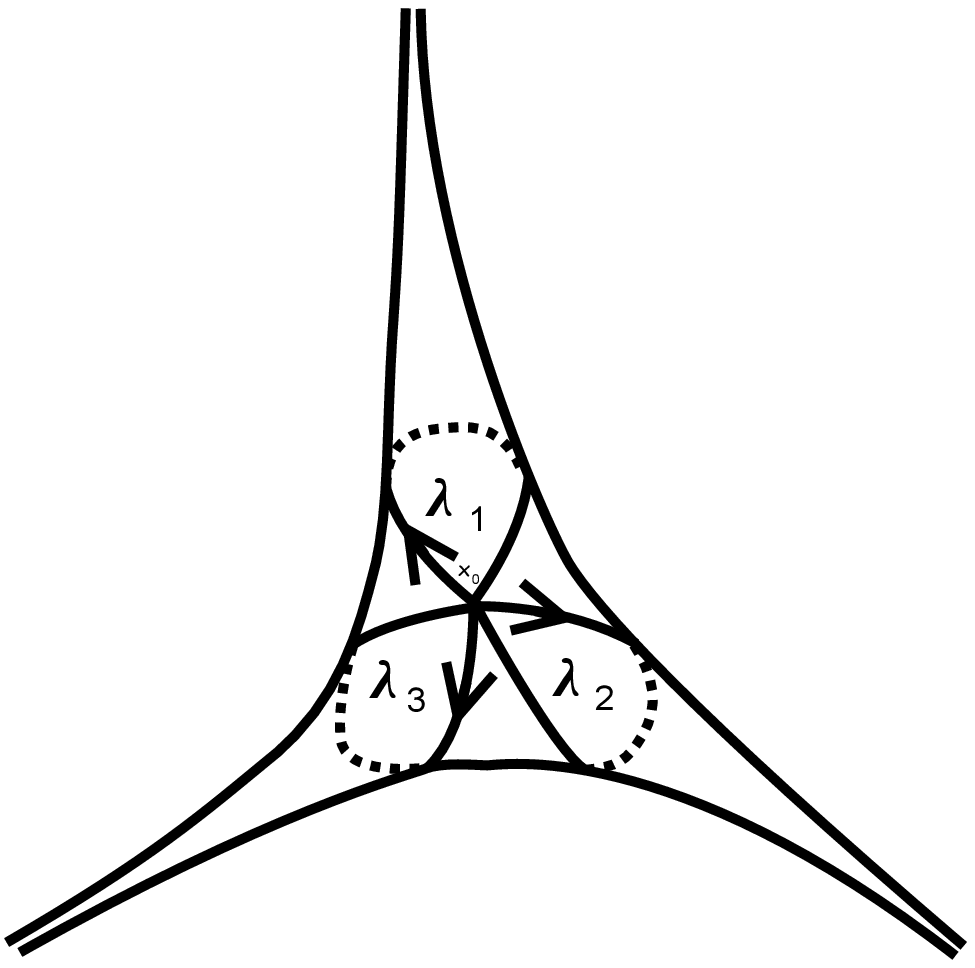, trim=-4cm 13cm 0cm 0cm, width=10cm}}
\caption{} \label{lacet}
\end{figure}

Les axes de $\g_2$ et $\g_3$ sont inclus dans le bord de $\C$. L'un des relevés  $\widetilde{l_3}$ du chemin $l_3$ converge en $-\infty$ vers le point $p_{\g_1}$ et en $+\infty$ vers le point $p^-_{\g_2}$. On peut donc supposer que le chemin $\widetilde{l_3}$ est le segment ouvert d'extrémité $p_{\g_1}$ et $p^-_{\g_2}$ qui est inclus dans $\C$. De la même façon, on peut supposer que $\widetilde{l_1}$ est le segment ouvert d'extrémités $p^-_{\g_2}$ et $p^-_{\g_3}$ inclus dans $\C$ et $\widetilde{l_2}$ le segment ouvert d'extrémités $p^-_{\g_3}$ et $p_{\g_1}$ inclus dans $\C$. La triangulation topologique de $\widetilde{S}$ peut donc être réalisée dans $\Omega$ l'intérieur de $\C$ par des géodésiques qui ne s'intersectent pas. On note $T_0$ l'unique triangle épointé inclus dans $\C$ défini par les points $p_{\g_1} , p^-_{\g_2} , p^-_{\g_3}$. Les composantes connexes de $\Omega$ privé des images des segments $\widetilde{l_1},\widetilde{l_2},\widetilde{l_3}$ sous l'action du groupe fondamental $\G$ de $\mathcal{P}$ sont des triangles ouverts. On note $T_i$ le triangle épointé qui borde $\widetilde{l_i}$, pour $i=1,...,3$.
\par{
L'action de $\G$ sur $\Omega$ admet pour domaine fondamental la réunion de $T_0$ et de n'importe quel $T_i$ pour $i=1,...,3$. Les identifications sont données par les éléments $\g_1,\g_2$ et $\g_3$.
}
\par{
La réunion des triangles épointés $T_0,T_1,T_2,T_3$ est un hexagone convexe car les triangles $T_0,T_1,T_2,T_3$ ont leurs sommets sur le bord de l'ouvert proprement convexe $\Omega$. Cet hexagone possède six côtés car, s'il possédait deux côtés consécutifs sur la même droite, alors la réunion de ces deux cotés serait incluse dans le bord de $\C$; or elle est incluse dans son intérieur puisqu'il s'agit de relevés des chemins $l_1,l_2,l_3$.
}
\\
\par{
On vient donc de construire une application de $\beta_f(P_{\delta_1,\delta_2,\delta_3})$ vers $\Q_{\delta_1,\delta_2,\delta_3}$. Le paragraphe suivant a pour but de construire l'application réciproque.
}
\\
\par{
La seule difficulté est de montrer que si l'on se donne un point de $\Q_{\delta_1,\delta_2,\delta_3}$, alors l'ensemble obtenu en prenant la réunion des images des triangles épointés $T_0,T_1,T_2,T_3$ sous l'action du groupe engendré par les éléments $\g_1,\g_2,\g_3$ est un ouvert proprement convexe $\Omega$.
}
\subsubsection{Un lemme de convexité}\label{convexe}

L'objet de ce paragraphe est de montrer la proposition suivante:

\begin{prop}\label{conv}
Soient $\delta_1,\delta_2,\delta_3 \in \s$ des éléments hyperboliques, quasi-hyperboliques ou paraboliques, et $(T_0,T_1,T_2,T_3,\g_1,\g_2,\g_3) \in \Q_{\delta_1,\delta_2,\delta_3}$, on note $\G$ le groupe engendré par $\g_1,\g_2,\g_3$ et $\Omega$ la réunion des triangles $\g T_0, \g T_1,\g T_2,\g T_3$ pour $\g \in \G$. L'ensemble $\Omega$ est un ouvert proprement convexe, le groupe $\G$ est isomorphe au groupe libre à deux générateurs et le quotient $\Quo$ est un pantalon projectif proprement convexe dont l'holonomie des lacets élémentaires est conjuguée à $\delta_1^{\pm},\delta_2^{\pm},\delta_3^{\pm}$.
\end{prop}

\begin{rem}
C'est cette partie de la démonstration du théorème \ref{module} qui diffère radicalement de la démonstration de Goldman pour le cas compact. Nous allons montrer cette proposition "à la main", alors que dans le cas compact, Goldman réussit à utiliser un résultat de Kozsul pour montrer cette proposition.
\end{rem}

Pour montrer cette proposition nous allons introduire un espace abstrait $X$ obtenu de la façon suivante: on considère $\widetilde{X}$ la réunion \underline{disjointe} des triangles $\g T_0, \g T_1,\g T_2,\g T_3$ pour $\g \in \G$.

On introduit une relation d'équivalence $\sim$ sur $\widetilde{X}$. Commençons par la définir sur les réunions disjointes $T_{i+2} \amalg \g_i T_{i+1}$, pour $i=1,...,3$ (les indices étant calculés modulo 3). Soient $x,y \in T_{i+2} \amalg \g_i T_{i+1}$, on a $x \sim y$ lorsque:

\begin{center}
($x\in T_{i+2}$ et $y \in \g_i T_{i+1}$ et $y =x$)
\end{center}

Soient $x,y \in \widetilde{X}$ on a $x \sim y$ lorsqu'il existe un $\g \in \G$ tel que $\g x \sim \g y$. On note $p$ l'application naturelle $p:X \rightarrow \P$, elle  permet de définir la notion de segment dans $X$.

\begin{defi}
Un \emph{segment dans $X$} est une application continue $s:[0,1] \rightarrow X$ telle que $p \circ s: [0,1] \rightarrow \P$ est une application injective dont l'image est un segment de longueur strictement inférieure à $\pi$, pour la métrique euclidienne canonique de $\P$.
\end{defi}

Les sous-ensembles de $X$ de la forme $\g T_0, \g T_1,\g T_2,\g T_3$ pour $\g \in \G$ seront appelés les \emph{cellules} de $X$. Leurs intérieurs seront appelés les \emph{cellules ouvertes} de $X$.

\begin{defi}
Une partie $C$ de $X$ est dite \emph{convexe} lorsque pour tout $x,y \in C$, il existe un segment d'extrémités $x$ et $y$ inclus dans $C$.
\end{defi}

Nous allons montrer le lemme suivant:

\begin{lemm}\label{coeur}
L'ensemble $X$ est convexe.
\end{lemm}

La démonstration de ce lemme est le coeur de la démonstration de la proposition \ref{conv}. On lui consacrera donc le paragraphe suivant. Mais commençons par montrer pourquoi ce lemme permet de conclure. Le lemme suivant est évident.

\begin{lemm}\label{local}
La restriction de $p$ à la réunion de deux cellules adjacentes est un homéomorphisme sur son image.
\end{lemm}

\begin{lemm}\label{inj}
La restriction de $p$ à toute partie convexe $C$ de $X$ est injective et l'image $p(C)$ est un convexe de $\P$.
\end{lemm}

\begin{proof}
Soient $x,y \in C$, il existe un segment dans $X$ d'extrémités $x$ et $y$; par conséquent l'image $p(C)$ est un convexe de $\P$. Si les points $x$ et $y$ vérifient $p(x) = p(y)$, alors $x=y$ puisque par définition un segment est une application injective.
\end{proof}

\begin{proof}[Démonstration de la proposition \ref{conv}]
Les lemmes \ref{coeur} et \ref{inj} montrent que l'application $p$ est injective. De plus, c'est un homéomorphisme local d'après le lemme \ref{local}. Par conséquent, c'est un homéomorphisme sur son image $\Omega$. L'ensemble $\Omega$ est donc un ouvert convexe.

Le groupe $\G$ préserve l'ouvert $\Omega$. Un domaine fondamental pour cette action est la réunion de deux cellules adjacentes. Et les identifications faites par $\G$ montrent que le quotient est un pantalon. Par conséquent, le groupe $\G$ est un groupe libre à 2 générateurs. Et l'holonomie des lacets élémentaires est conjuguées à $\delta_1^{\pm},\delta_2^{\pm},\delta_3^{\pm}$.

Il  reste à montrer que l'ouvert $\Omega$ est proprement convexe. Si l'ouvert $\Omega$ n'était pas proprement convexe alors le groupe $\G$ préserverait un point ou une droite de $\P$. Nous allons montrer que les droites et les points fixes de $\g_1$ ne sont pas préservés par $\g_2$ et $\g_3$. Pour simplifier la discussion on ne fait que le cas où les $\g_i$ sont hyperboliques. Les autres cas sont des "dégénérescences" de ce cas. Les droites des côtés de $T_0$ définissent un pavage en 4 triangles fermés de $\P$. Le point $p^+_{\g_i}$ pour $i=1,...,3$ ne peut pas appartenir au triangle fermé dont l'intersection avec le triangle fermé $\overline{T_0}$ est le point $p^-_{\g_i}$ et le côté opposé au point $p^-_{\g_i}$, car sinon $p^-_{\g_i} \in \Omega$. Par conséquent, les points $(p^+_{\g_i})_{i=1,...,3}$ sont dans les intérieurs des triangles de ce pavage. De plus, il est facile de voir que si $i \neq j$, alors $p^+_{\g_i}$ et $p^+_{\g_j}$ ne sont pas dans le même triangle. On note $T_{\g_i}$ le triangle ouvert avec $p^+_{\g_i} \in T_{\g_i}$, pour $i=1,...,3$. Le point $p^0_{\g_i}$ est sur la droite tangente au bord $\partial \Omega$ de $\Omega$ au point $p^-_{\g_i}$, par conséquent le point $p^0_{\g_i}$ appartient aussi au triangle $T_{\g_i}$. Ainsi, il n'y a aucun point et aucune droite de $\P$ qui sont fixés par les trois éléments $\g_1,\g_2,\g_3$.
\end{proof}

\subsubsection{Démonstration du lemme \ref{coeur}}

Les composantes connexes du bord des cellules de $X$ seront appelées les \emph{murs} de $X$. Soient $C_1, C_2$ deux cellules de $X$, on note $\C_{C_1,C_2} = \{ (x,y) \in C_1 \times C_2 \,|\, \textrm{ il existe un segment dans } $X$ \textrm{ d'extrémités } x \textrm{ et } y \}$. Nous allons montrer que $\C_{C_1,C_2}$ est ouvert et fermé dans $C_1 \times C_2$. Le point difficile étant la fermeture.

\begin{lemm}\label{ouvert}
 $\C_{C_1,C_2}$ est ouvert dans $C_1 \times C_2$.
\end{lemm}

\begin{proof}
Soit $(x,y) \in \C_{C_1,C_2}$, il existe un segment $s$ reliant $x$ à $y$. Comme l'image d'un segment est compact elle est incluse dans un nombre fini de cellules de $X$. On note $(x^i)_{i=1...N}$ les points d'intersections du segment $[x,y]$ avec les murs de $X$ numérotés via la paramétrisation de $[x,y]$. On pose $x^0=x$ et $x^{N+1}=y$. Comme chaque cellule est convexe, il existe des voisinages $V_{x^i}$ de $x^i$ (pour $i=1,...,N+1$) dans $X$ tels que l'enveloppe convexe dans chaque cellule des couples $(V_{x^i},V_{x^{i+1}})$ contienne un voisinage convexe du segment $[x^i,x^{i+1}]$ inclus dans la cellule contenant $[x^i,x^{i+1}]$. Comme la réunion de deux cellules adjacentes est convexe, la réunion de ces voisinages contient un voisinage convexe de $[x,y]$.
\end{proof}

On souhaite à présent montrer le lemme suivant.

\begin{lemm}\label{ferme}
 $\C_{C_1,C_2}$ est fermé dans $C_1 \times C_2$.
\end{lemm}

Les lemmes \ref{inj} et \ref{local}  montrent que $p$ est un homéomorphisme local. On peut donc munir $X$ de la métrique riemanienne induite par la métrique riemanienne canonique de $\P$. Le complété de $X$ pour cette distance sera noté $\overline{X}$.

\begin{lemm}\label{infini}
Soit $(s_n)_{n \in \N}$ une suite de segments d'extrémités $x_n \in C_1$ et $y_n \in C_2$, le segment $s_n$ traverse $N_n$ murs $M_1^n,...,M^n_{N_n}$ de $X$ ( ordonnés par la paramétrisation de $s_n$). On note $(x_n^{i})_{i=1...N_n}$ les points d'intersection de $s_n$ avec les murs de $X$ ( ordonnés via la paramétrisation de $s_n$). Si la suite $(x_n)_{n \in \N}$ converge dans $X$ et la suite $(x_n^1)_{n \in \N}$ diverge dans X et converge dans $\overline{X}$ vers $x^1_{\infty}$, alors, si $n$ est assez grand, la suite $N_n$ est constante égale à $N$, les suites $M^n_1,...,M^n_{N}$ sont constantes et il existe un $\g \in \G$ et une cellule $C'$ adjacente à $C_1$ tels que $s_n \subset \underset{n \in \N}{\bigcup} \g^n(C_1 \cup C')$ et $y_n$ tend vers $x_{\infty}^1$.
\end{lemm}

\begin{proof}
Quitte à extraire on peut supposer que les suites $M_1^n$ et $M_2^n$ sont constantes car les cellules de $X$ ont un nombre fini de côtés. On note $C'$ la cellule adjacente à $C_1$ et telle que $C_1 \cap C' = M_1^n$ et $C"$ la cellule adjacente à $C'$ et telle que $C' \cap C" = M_2^n$. Il existe un élément $\g \in \G$ tel que $\g C_1 = C"$.

Comme $\g$ est hyperbolique ou quasi-hyperbolique ou parabolique, l'ensemble $K=\underset{n \in \N}{\bigcup} \g^n(C_1 \cup C')$ est un convexe de $X$ (voir la figure \ref{lemp}).

\begin{figure}[!h]
\centerline{\psfig{figure=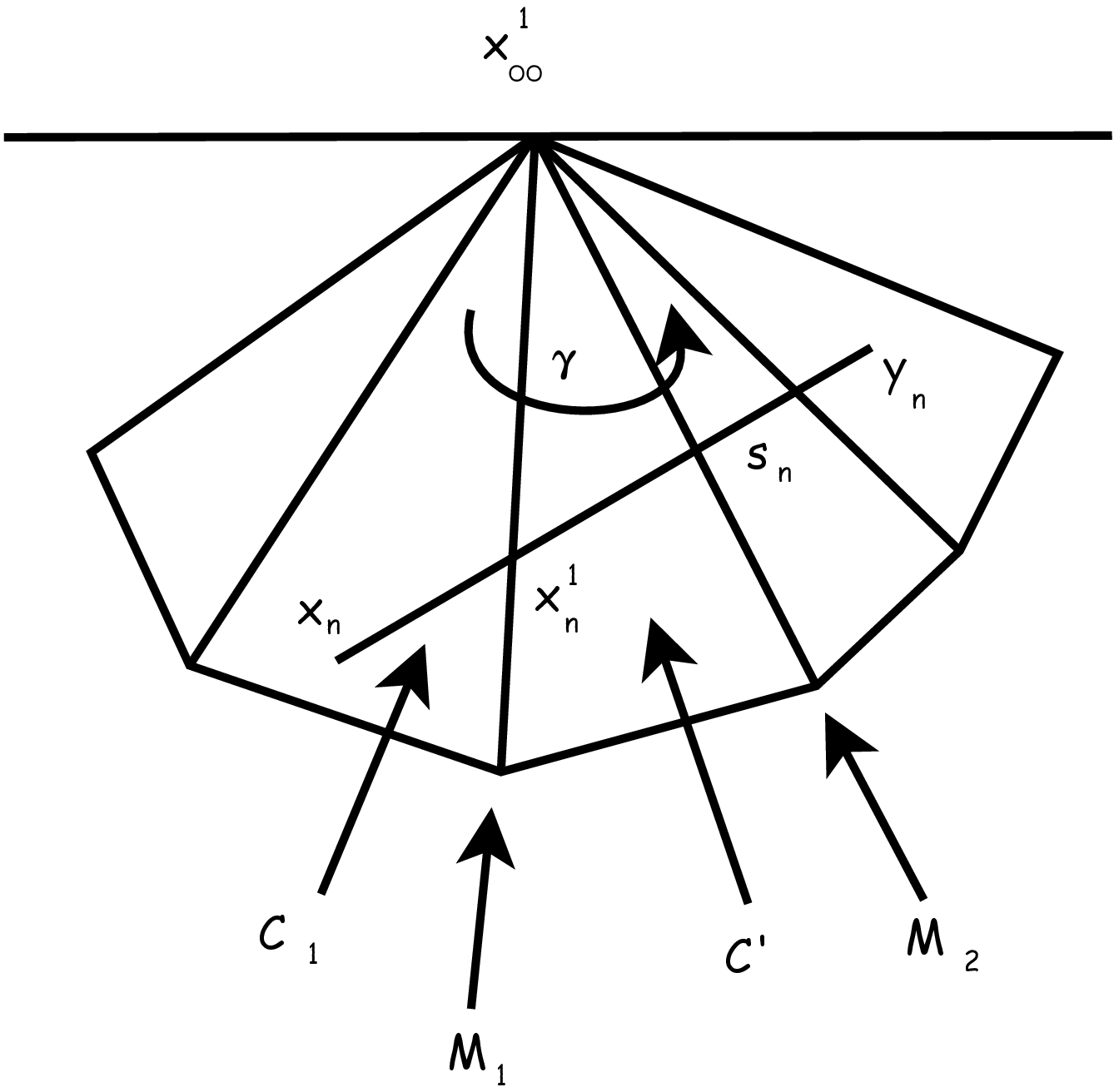, trim=0cm 11cm 0cm 0cm, width=6cm}
\hspace{3em}
\psfig{figure=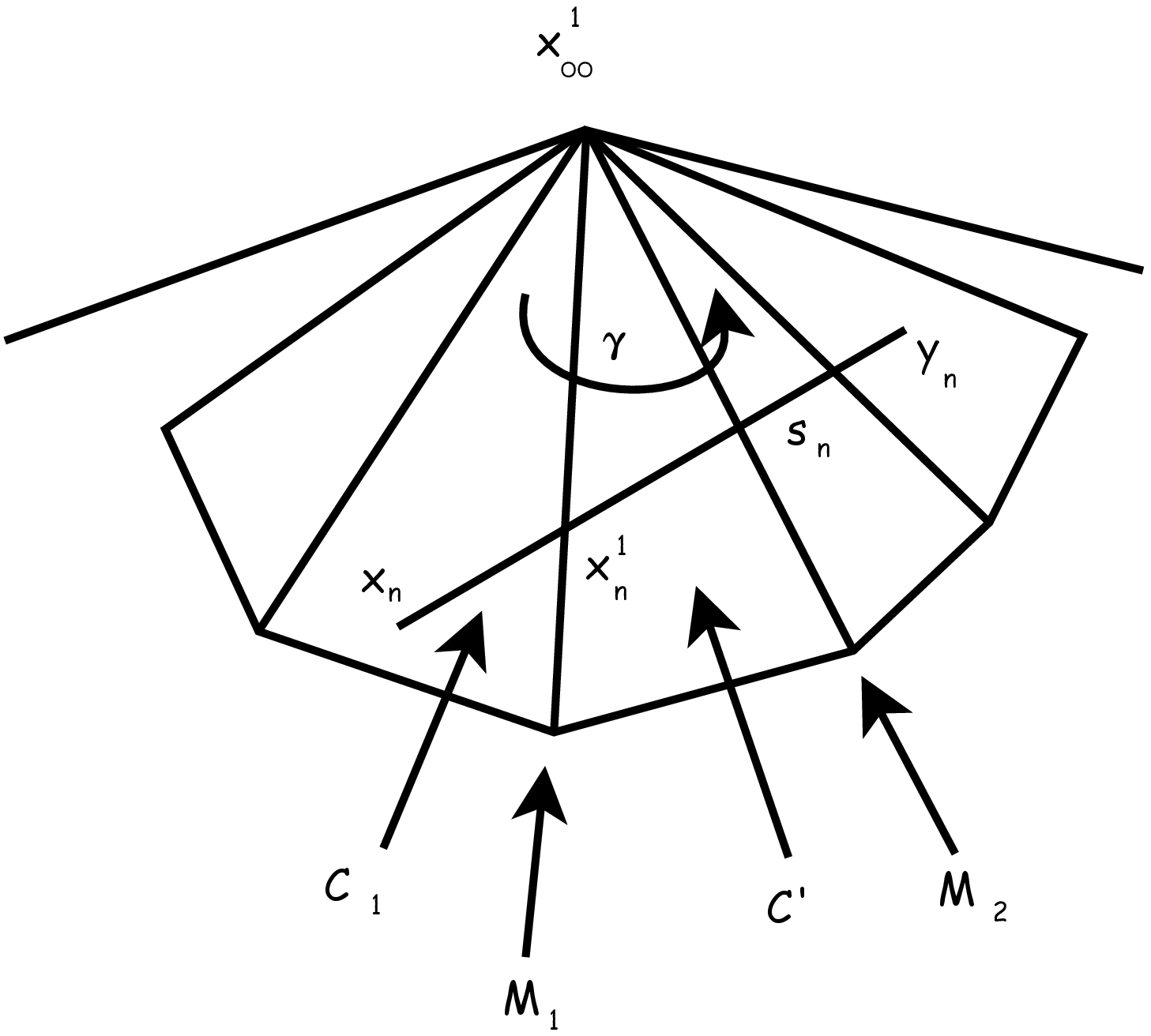, trim=0cm 11cm 0cm 0cm, width=6cm}}
\caption{Le cas parabolique et le cas hyperbolique ou quasi-hyperbolique} \label{lemp}
\end{figure}

On note $M^{2i-1} = \g^{i-1} M^1$ et $M^{2i}=\g^{i-1} M^2$, pour $i \geqslant 1$.  Il est clair que si $n$ est assez grand alors $x_n^i \in M^i$, pour $i \geqslant 1$ et $x_n^i \in M^i$ converge vers $x^1_{\infty}$. Donc la cellule $C_2$ est de la forme $\g^n C_1$ ou $\g^n C'$ pour un $n\in \N$. La suite $N_n$ est donc constante égale à un certain $N$ et les suites $M^n_1,...,M^n_{N}$ sont constantes et $y_n$ tend vers $x_{\infty}^1$.
\end{proof}

On peut à présent montrer le lemme \ref{ferme}.

\begin{proof}
[Démonstration du lemme \ref{ferme}]

Soit $(x_n,y_n) \in \C_{C_1,C_2}$ une suite qui converge dans $C_1 \times C_2$ vers le couple $(x,y)$. Il existe un segment $s_n$ dans $X$ reliant $x_n$ à $y_n$. On note $(x_n^{i})_{i=1...N_n}$ les points d'intersection de $s_n$ avec les murs de $X$ ( numérotés via la paramétrisation de $s_n$ ). Nous allons montrer par l'absurde que les suites $(x_n^{i})_{n \in \N}$ convergent dans $X$.

On a deux cas à distinguer:
\begin{itemize}
\item Toutes les suites $(x_n^{i})_{n \in \N}$ divergent. En particulier, la suite $(x_n^1)_{n \in \N}$ diverge dans $X$, quitte à extraire, on peut supposer qu'elle converge dans $\overline{X}$; et la suite $(x_n)_{n \in \N}$ converge dans $X$.

\item Il existe un $i_0=2,...,N$ tel que la suite $(x_n^{i_0})_{n \in \N}$ diverge dans $X$ mais sous-converge dans $\overline{X}$ et la suite $(x_n^{i_0-1})_{n \in \N}$ converge dans $X$.
\end{itemize}

Dans le premier cas, le lemme \ref{infini}  montre que la suite $(y_n)_{n \in \N}$ diverge, ce qui est absurde. Dans le second cas, le lemme \ref{infini} appliqué aux segments $[x_n^{i_0-1},y_n]$ montre que la  suite $(y_n)_{n \in \N}$ diverge. Ce qui est absurde. Donc, les suites $(x_n^{i})_{n \in \N}$ convergent. Par conséquent, les points $x$ et $y$ sont reliés par une réunion finie de segments qui vérifie de plus que sa restriction à la réunion de deux cellules adjacentes est un segment. Ce chemin est donc un segment.
\end{proof}

On peut à présent montrer le lemme \ref{coeur}

\begin{proof}
[Démonstration du lemme \ref{coeur}]
Les lemmes \ref{ouvert} et \ref{ferme}  montrent que si $\C_{C_1,C_2} \neq \varnothing$ alors  $\C_{C_1,C_2} = C_1 \times C_2$. On peut définir le graphe $\A$ dont les sommets sont les cellules de $X$ et deux sommets de $\A$ sont reliés si les cellules correspondantes bordent un même mur. Une simple récurrence sur la distance dans le graphe $\A$ à une cellule de départ $C_0$ montre que, pour tout $n$, la réunion des cellules à distance inférieure ou égale à $n$ de $C_0$ est une partie convexe. On pourra remarquer que le graphe $\A$ est l'arbre infini de valence 3.
\end{proof}

\subsubsection{L'espace $\Q_{\delta_1,\delta_2,\delta_3}$ est homéomorphe à $\R^2$}\label{computation}

\begin{prop}\label{modQ}
Soient $\delta_1,\delta_2,\delta_3$ trois éléments hyperboliques, quasi-hyperboliques ou paraboliques de $\s$, l'espace $\Q_{\delta_1,\delta_2,\delta_3}$ est homéomorphe à $\R^2$.
\end{prop}

La démonstration de cette proposition est l'objet des deux prochains paragraphes. Cette démonstration est une légère généralisation de celle de Goldman (\cite{Gold1}) qui ne traite pas les cas où les $\delta_i$ sont paraboliques ou quasi-hyperboliques.

\paragraph{Paramétrisation de l'hexagone}

On commence par introduire l'espace suivant:

$$
\mathcal{H'}=
\left\{
\begin{array}{l}
(T_0,T_1,T_2,T_3) \textrm{ tel que:}\\
\\
1.\, T_0,T_1,T_2,T_3 \textrm{ sont des triangles épointés et le triangle } T_0\\
\textrm{intersecte } T_1,T_2,T_3 \textrm{ le long de ces 3 arêtes}\\
\\
2. \underset{i=0,...,3}{\bigcup} \overline{T_i} \textrm{ est un hexagone }\\
\textrm{ convexe a exactement 6 côtés}\\
\end{array}
\right\}
$$

Le groupe $\s$ agit proprement et librement sur $\mathcal{H'}$, on note $\mathcal{H}$ le quotient $\mathcal{H'}/_{\s}$. Nous allons montrer le lemme suivant.

\begin{lemm}
L'espace $\H$ est homéomorphe à $\R^4$.
\end{lemm}

\begin{proof}
Les notations sont les mêmes que celles de la figure \ref{hexa}. Commençons par nommer les sommets du triangle $T_0$. On notera $p_1$ l'intersection de $\overline{T_0},\overline{T_2}$ et $\overline{T_3}$, $p_2$ l'intersection de $\overline{T_0},\overline{T_3}$ et $\overline{T_1}$ et $p_3$ l'intersection de $\overline{T_0},\overline{T_1}$ et $\overline{T_2}$. On peut supposer que leurs coordonnées sont données par:

$$
\left\{
\begin{array}{ccc}
p_1 & = & [1:0:0]\\
p_2 & = & [0:1:0]\\
p_3 & = & [0:0:1]\\
\end{array}
\right.
$$

On notera $q_1$ le sommet de $T_1$ qui n'intersecte pas $\overline{T_0}$, $q_2$ le sommet de $T_2$ qui n'intersecte pas $\overline{T_0}$ et $q_3$ le sommet de $T_3$ qui n'intersecte pas $\overline{T_0}$.

$$
\left\{
\begin{array}{ccc}
q_1 & = & [-1:b_1:c_1]\\
q_2 & = & [a_2:-1:c_2]\\
q_3 & = & [a_3:b_3:-1]\\
\end{array}
\right.
$$

Où, les quantités $b_1,c_1,a_2,c_2,a_3,b_3$ sont strictement positives. Le stabilisateur de $p_1,p_2,p_3$ dans $\s$ est le groupe $D$ des matrices diagonales à diagonale strictement positive. On considère l'élément $g$ donné par la matrice suivante:

$$
g=\left(
\begin{array}{ccc}
\lambda & 0   & 0\\
0       & \mu & 0\\
0       & 0   & \nu\\
\end{array}
\right)
$$

Où, les quantités $\lambda,\mu,\nu$ sont strictement positives. L'action de $g$ sur $\H'$ fixe les points $(p_i)_{i=1,...,3}$. Son action sur les points $(q_i)_{i=1,...,3}$ s'écrit de la façon suivante:

$$
\left\{
\begin{array}{ccc}
b_1 & \mapsto & \frac{\mu}{\lambda} b_1\\
c_1 & \mapsto & \frac{\nu}{\lambda} c_1\\
a_2 & \mapsto & \frac{\lambda}{\mu} a_2\\
c_2 & \mapsto & \frac{\nu}{\mu} c_2\\
a_3 & \mapsto & \frac{\lambda}{\nu} a_3\\
b_3 & \mapsto & \frac{\mu}{\nu} b_3
\end{array}
\right.
$$

On notera $\rho_i$ le birapport des 4 droites concourantes en $p_i$ définies par les triangles $T_0,T_1,T_2,T_3$. On a:

$$
\left\{
\begin{array}{ccc}
\rho_1 & = & b_3 c_2\\
\rho_2 & = & c_1 a_3\\
\rho_3 & = & a_2 b_1\\
\end{array}
\right.
$$

L'hexagone $\underset{i=0,...,3}{\bigcup} \overline{T_i}$ est convexe qui possède exactement six côtés si et seulement si pour tout $i=1,...,3, \, \rho_i > 1$. On définit en plus les quantités suivantes qui sont $D$-invariantes:

\begin{figure}
\begin{center}
\includegraphics[width=8cm]{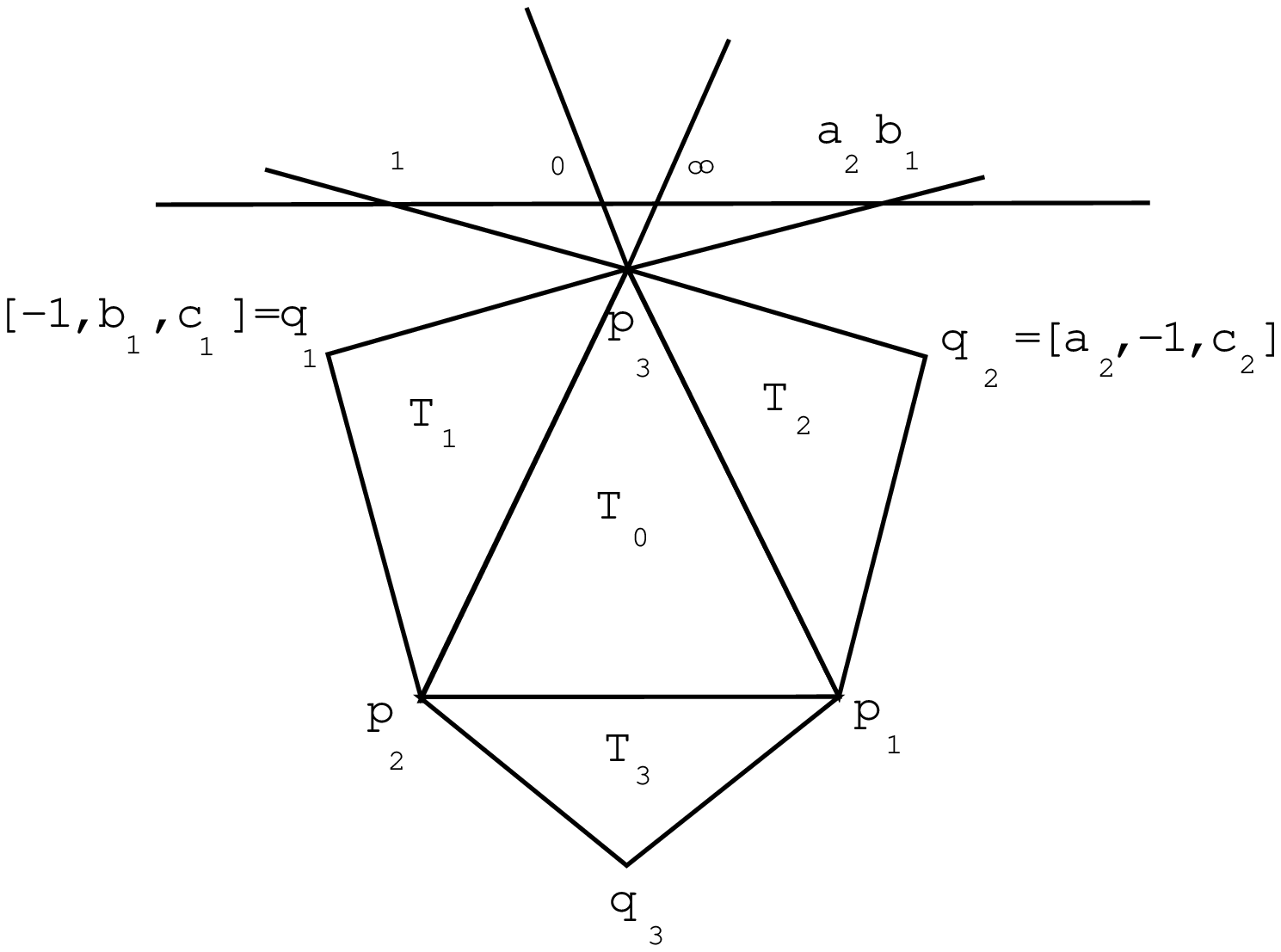}
\end{center}
\caption{} \label{hexastrict}
\end{figure}

$$
\left\{
\begin{array}{ccc}
\sigma_1 & = & a_2 b_3 c_1\\
\sigma_2 & = & a_3 b_1 c_2\\
\end{array}
\right.
$$

Elles vérifient: $\sigma_1 \sigma_2 = \rho_1 \rho_2 \rho_3$. C'est ce qu'il fallait montrer. On a même le lemme plus précis suivant.
\end{proof}

\begin{lemm}
Il existe des applications continues $\rho_1, \rho_2, \rho_3,\sigma_1, \sigma_2$ de l'espace $\H$ vers $\R$ telles que l'application $(\rho_1, \rho_2, \rho_3, \sigma_1, \sigma_2) : \H \rightarrow \mathcal{I}$ est un homéomorphisme, où $\mathcal{I} = \{ (\rho_1, \rho_2, \rho_3, \sigma_1,$ \\ $\sigma_2) \in \R^5 \, |\,  \rho_1, \rho_2, \rho_3 > 1, \, \sigma_1, \sigma_2 > 0, \, \sigma_1 \sigma_2 = \rho_1 \rho_2 \rho_3 \}$
\end{lemm}

\paragraph{Paramétrisation des éléments $\g_i$}

Nous allons avoir besoin du lemme suivant.

\begin{lemm}\label{elimination}
Soient $T$ et $T'$ deux triangles ouverts disjoints de $\P$ qui ont un sommet $p$ en commun et un élément $\g$ de $\s$ qui fixe $p$, tel que $\g T = T'$. Si le spectre de $\g$ est réel positif alors $\g$ est hyperbolique ou quasi-hyperbolique ou parabolique.
\end{lemm}

\begin{proof}
Il suffit de montrer que $\g$ ne peut pas être conjugué à l'une des deux matrices suivantes:
$$
A=
\begin{array}{ccc}
\left(
\begin{array}{ccc}
\alpha &    0      & 0\\
0      & \alpha    & 0\\
0      &  0        & \alpha^{-2}
\end{array}
\right)

&
B=
\left(
\begin{array}{ccc}
1      & 1   & 0\\
0      & 1   & 0\\
0      & 0   & 1
\end{array}
\right)
\end{array}
$$
où $\alpha > 1$.

\begin{figure}
\begin{center}
\includegraphics[trim=5.5cm 14cm 0cm 0cm, width=8cm]{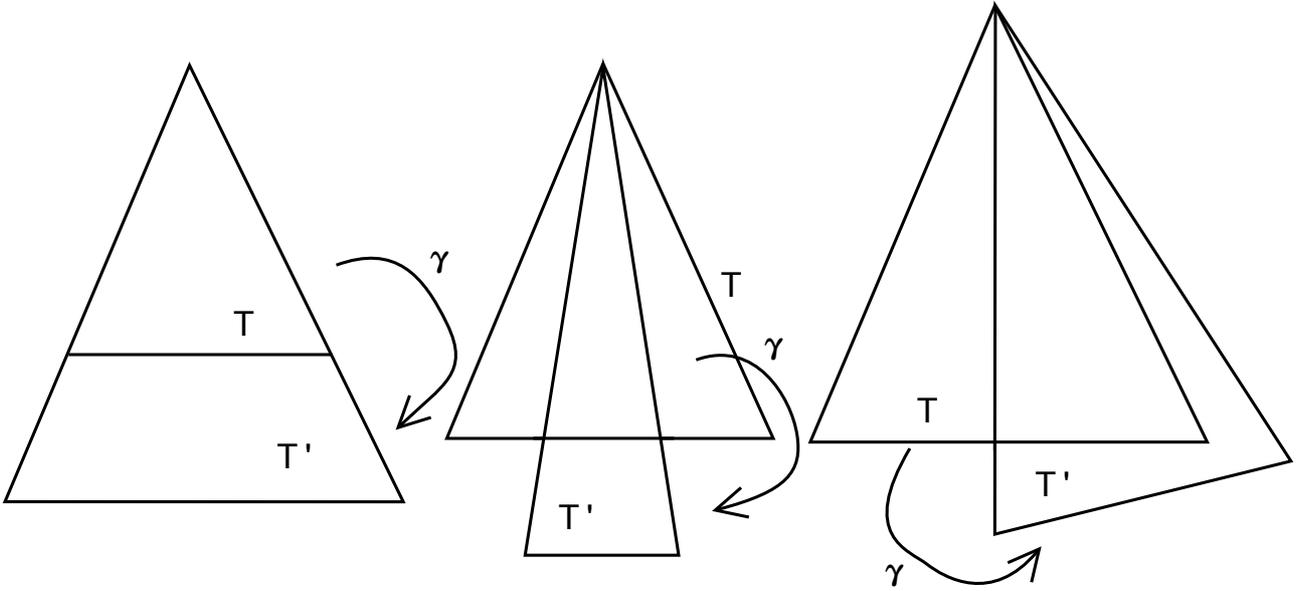}
\end{center}
\caption{Toutes les situations possibles si $\g$ est conjugué à $A$ ou $B$} \label{hexastrict}
\end{figure}

Il est facile de voir que l'image de tout triangle ouvert $T$ par un élément conjugué à l'une de ces deux matrices qui fixe un sommet de $T$ est un triangle $T'$ tel que l'intersection $T \cap T'$ est non vide.
\end{proof}

On peut à présent montrer la proposition suivante:

\begin{prop}\label{objmod}
Il existe des applications continues $\rho_1, \rho_2, \rho_3, \sigma_1, \sigma_2,\kappa_1,\mu_1,\nu_1,\kappa_2,\mu_2,\nu_2,$ \\ $\kappa_3,\mu_3,\nu_3$ de l'espace $\Q_{\delta_1,\delta_2,\delta_3}$ vers $\R$ telles que l'application $(\rho_1, \rho_2, \rho_3, \sigma_1, \sigma_2, \kappa_1,\mu_1,\nu_1,\kappa_2,\mu_2,\nu_2,$ \\$\kappa_3,\mu_3,\nu_3) : \Q_{\delta_1,\delta_2,\delta_3} \rightarrow \mathcal{E}$ est un homéomorphisme, où
$$
\mathcal{E} =
\left\{
\begin{array}{l}
(\rho_1, \rho_2, \rho_3, \sigma_1, \sigma_2, \kappa_1,\mu_1,\nu_1,\kappa_2,\mu_2,\nu_2,\kappa_3,\mu_3,\nu_3) \in \R^{14} \textrm{ tel que:}\\
\\
\rho_1, \rho_2, \rho_3 > 1,\, \sigma_1, \sigma_2 > 0\\
\kappa_1,\mu_1,\nu_1,\kappa_2,\mu_2,\nu_2,\kappa_3,\mu_3,\nu_3 > 0\\
\\
1.\, \sigma_1 \sigma_2 = \rho_1 \rho_2 \rho_3\\
2.\, \kappa_1 \mu_1 \nu_1 = \kappa_2 \mu_2 \nu_2 = \kappa_3 \mu_3 \nu_3 = 1\\
3.\, \kappa_1 \nu_2 \mu_3 = \kappa_2 \nu_3 \mu_1 = \kappa_3 \nu_1 \mu_2 = 1\\
\\
4.\, \kappa_i = \lambda(\delta_i), \textrm{ pour } i=1,...,3.\\
\\
5.\, -\mu_i+\nu_i(\rho_i -1) = \tau(\delta_i), \textrm{ pour } i=1,...,3.\\
\end{array}
\right\}.
$$
\end{prop}

\begin{proof}
On se donne $(T_0,T_1,T_2,T_3) \in \mathcal{H}$ et on cherche $\g_1$ (resp. $\g_2$ resp. $\g_3$) conjuguée à $\delta_1$ (resp. $\delta_2$ resp. $\delta_3$) telle que $\g_1(T_2) = T_3, \, \g_2(T_3) = T_1, \, \g_3(T_1) = T_2$ et tel que $\g_3 \g_2 \g_1 =1$. Nous allons avoir besoin de travailler dans $\R^3$ plutôt que dans $\P$. On notera $e_i$ un relevé de $p_i$ et $f_i$ un relevé de $q_i$, pour $i=1,...,3$. On peut supposer que l'on a:

$$
\begin{array}{cc}
\left\{
\begin{array}{ccc}
e_1 & = & (1,0,0)\\
e_2 & = & (0,1,0)\\
e_3 & = & (0,0,1)\\
\end{array}
\right.

&

\left\{
\begin{array}{ccc}
f_1 & = & (-1,b_1,c_1)\\
f_2 & = & (a_2,-1,c_2)\\
f_3 & = & (a_3,b_3,-1)\\
\end{array}
\right.
\end{array}
$$

Il existe donc des réels $\kappa_1,\mu_1,\nu_1,\kappa_2,\mu_2,\nu_2,\kappa_3,\mu_3,\nu_3 > 0$ tels que:

$$
\begin{array}{cc cc c}
\left\{
\begin{array}{ccc}
\g_1(e_1) & = & \kappa_1 e_1\\
\g_1(f_2) & = & \mu_1 e_2\\
\g_1(e_3) & = & \nu_1 f_3\\
\end{array}
\right.
&
\,\,\,
&
\left\{
\begin{array}{ccc}
\g_2(e_2) & = & \kappa_2 e_2\\
\g_2(f_3) & = & \mu_2 e_3\\
\g_2(e_1) & = & \nu_2 f_1\\
\end{array}
\right.
&
\,\,\,
&
\left\{
\begin{array}{ccc}
\g_3(e_3) & = & \kappa_3 e_3\\
\g_3(f_1) & = & \mu_3 e_1\\
\g_3(e_2) & = & \nu_3 f_2\\
\end{array}
\right.
\end{array}
$$

On peut donc écrire les matrices de $\g_1,\g_2,\g_3$ dans la base $(e_1,e_2,e_3)$:

$$
\begin{array}{ccc ccc}
\g_1 & = &
\left(
\begin{array}{ccc}
\kappa_1 &  \nu_1 c_2 a_3 + \kappa_1 a_2  &  \nu_1 a_3 \\
 0       &  \nu_1 c_2 b_3 - \mu_1         &  \nu_1 b_3 \\
 0       & -\nu_1 c_2                     & -\nu_1     \\
\end{array}
\right)
&
\g_2 & = &
\left(
\begin{array}{ccc}
-\nu_2     & 0        & -\nu_2 a_3                    \\
 \nu_2 b_1 & \kappa_2 &  \nu_2 a_3 b_1 + \kappa_2  b_3 \\
 \nu_2 c_1 & 0        &  \nu_2 a_3 c_1 - \mu_2       \\
\end{array}
\right)
\end{array}
$$

$$
\begin{array}{ccc}
\g_3 & = &
\left(
\begin{array}{ccc}
 \nu_3 b_1 a_2  - \mu_3         &  \nu_3 a_2   & 0 \\
-\nu_3 b_1                      & -\nu_3       & 0 \\
 \nu_3 b_1 c_2  + \kappa_3 c_1  &  \nu_3 c_2   & \kappa_3 \\
\end{array}
\right)
\end{array}
$$

On peut calculer le déterminant de $\g_1,\g_2,\g_3$, on trouve:
$$\det(\g_1)= \kappa_1 \mu_1 \nu_1,\, \det(\g_2)= \kappa_2 \mu_2 \nu_2 ,\, \det(\g_3)= \kappa_3 \mu_3 \nu_3$$

De plus, il est clair que le produit $\g_3\g_2\g_1$ vérifie:

$$\g_3\g_2\g_1(e_1) = (\kappa_1 \nu_2 \mu_3)\, e_1,\, \g_3\g_2\g_1(f_2) = (\kappa_2 \nu_3 \mu_1)\, f_2 ,\, \g_3\g_2\g_1(e_3) = (\kappa_3 \nu_1 \mu_2)\, e_3$$

Il vient que le produit $\g_3\g_2\g_1$ vérifie $\g_3\g_2\g_1=1$ si et seulement si $\kappa_1 \nu_2 \mu_3 = \kappa_2 \nu_3 \mu_1 = \kappa_3 \nu_1 \mu_2 = 1$.

Il  reste à comprendre la condition "$\g_i$ est conjuguée à $\delta_i$" pour $i=1,...,3$. Nous allons montrer que l'élément $\g_i$ est conjugué à $\delta_i$ pour $i=1,...,3$ si et seulement si $\kappa_i = \lambda(\delta_i)$  et $-\mu_i+\nu_i(\rho_i -1) = \tau(\delta_i)$ pour $i=1,...,3$.

Supposons que $\kappa_i = \lambda(\delta_i)$  et $-\mu_i+\nu_i(\rho_i -1) = \tau(\delta_i)$ pour $i=1,...,3$. Alors un calcul facile montre que le spectre de $\g_i$ est réel positif. Par conséquent, le lemme \ref{elimination}  montre que $\g_i$ est hyperbolique ou quasi-hyperbolique ou parabolique. La proposition \ref{conv}  montre que les éléments $\g_1$, $\g_2$ et $\g_3$ préservent un ouvert proprement convexe. La proposition \ref{consclassi} montre que l'élément $\g_i$ est conjugué à l'élément $\delta_i$ pour $i=1,...,3$ puisque $\lambda(\gamma_i) = \kappa_i = \lambda(\delta_i)$ et $\tau(\g_i)=-\mu_i+\nu_i(\rho_i -1) = \tau(\delta_i)$.

Supposons à présent que $\g_i$ est conjuguée à $\delta_i$ pour $i=1,...,3$ alors clairement $\lambda(\gamma_i) = \kappa_i = \lambda(\delta_i)$ et $\tau(\g_i)=-\mu_i+\nu_i(\rho_i -1) = \tau(\delta_i)$ pour $i=1,...,3$.
\end{proof}

\begin{proof}[Démonstration de la proposition \ref{modQ}]
La proposition \ref{objmod} montre qu'il ne  reste plus qu'à montrer que $\mathcal{E}$ est homéomorphe à $\R^2$. On pose $\lambda_i = \lambda(\delta_i)$ et $\tau_i = \tau(\delta_i)$, pour i=1,...,3. On commence par résoudre les équations (2.), (3.), (4.), elles forment un système de 6 équations qui se ramène facilement à un système d'équations linéaire de rang 5, on obtient:

$$
\begin{array}{lll}
\kappa_1 = \lambda_1                                         & \mu_1 = \sqrt{\frac{\lambda_3}{\lambda_1 \lambda_2}} s          & \nu_1 = \sqrt{\frac{\lambda_2}{\lambda_1 \lambda_3}} s^{-1} \\

\nu_2  = \sqrt{\frac{\lambda_3}{\lambda_2 \lambda_1}} s^{-1} & \kappa_2 = \lambda_2                                            & \mu_2 = \sqrt{\frac{\lambda_1}{\lambda_2 \lambda_3}} s \\

\mu_3 = \sqrt{\frac{\lambda_2}{\lambda_3 \lambda_1}} s       & \nu_3  = \sqrt{\frac{\lambda_1}{\lambda_3 \lambda_2}} s^{-1}    & \kappa_3 = \lambda_3\\
\end{array}
$$

Avec un paramètre $s$ qui vérifie $s>0$. On peut à présent résoudre les équations (5.), on obtient pour $i=1...3$:

$$
\rho_i = 1+ \frac{\tau_i+\gamma_i}{\beta_i} = 1 + \frac{\tau_i + \sqrt{\frac{\lambda_{i+2}}{\lambda_{i}\lambda_{i+1}}}s}{\sqrt{\frac{\lambda_{i+1}}{\lambda_{i}\lambda_{i+2}}}s^{-1}}= 1+\sqrt{\frac{\lambda_{i}\lambda_{i+2}}{\lambda_{i+1}}} \tau_i s + \frac{\lambda_{i+2}}{\lambda_{i+1}} s^2
$$

Pour finir, il faut résoudre (1.), on obtient:

$$\sigma_1= t \rho_2 \textrm{ et } \sigma_2 = \frac{\rho_1 \rho_3}{t}$$

Avec un paramètre $t$ qui vérifie $t>0$.
\end{proof}

La proposition \ref{panta-combi}  montre que $\beta_f(P_{\delta_1,\delta_2,\delta_3})$ est homéomorphe à $\Q_{\delta_1,\delta_2,\delta_3}$, on a donc le corollaire suivant:

\begin{coro}\label{modpant}
Soient $\delta_1,\delta_2,\delta_3 \in \s$ des éléments hyperboliques, quasi-hyperboliques ou paraboliques alors l'espace $\beta_f(P_{\delta_1,\delta_2,\delta_3})$ est homéomorphe à $\R^2$.
\end{coro}

Ceci conclut la démonstration du cinquième point du théorème \ref{module} et par conséquent sa démonstration aussi.

\section{Composantes connexes d'espace de représentations}

\subsection{Préliminaires}

\subsubsection{Le cas compact}

Le but de cette partie est de montrer que les structures projectives proprement convexes sont en fait des objets très naturels. Précisons notre pensée en donnant le théorème suivant qui montre que les structures projectives proprement convexes sur les surfaces compacts sont des objets naturels.

\begin{theo}[Koszul-Choi-Goldman]
L'espace des modules des structures projectives proprement convexes sur une surface compacte $S$ est une composante connexe de l'espace des représentations, à conjugaison près, du groupe fondamental de $S$ dans $\s$.
\end{theo}

L'ouverture de l'espace des modules des structures projectives proprement convexes sur une surface compacte a été démontrée par Koszul dans \cite{Kos}. Choi et Goldman ont montré dans \cite{ChGo} que cet espace est fermé. Goldman a montré dans \cite{Gold1} que cet espace est connexe.

\subsubsection{Espaces de représentations}

Nous allons montrer un résultat analogue dans le cas non compact. Comme le groupe fondamental d'une surface de type fini non compacte est un groupe libre, on ne peut pas espérer avoir ce genre de résultat en prenant l'ensemble des représentations en entier.

\begin{figure}[!h]
\centerline{\psfig{figure=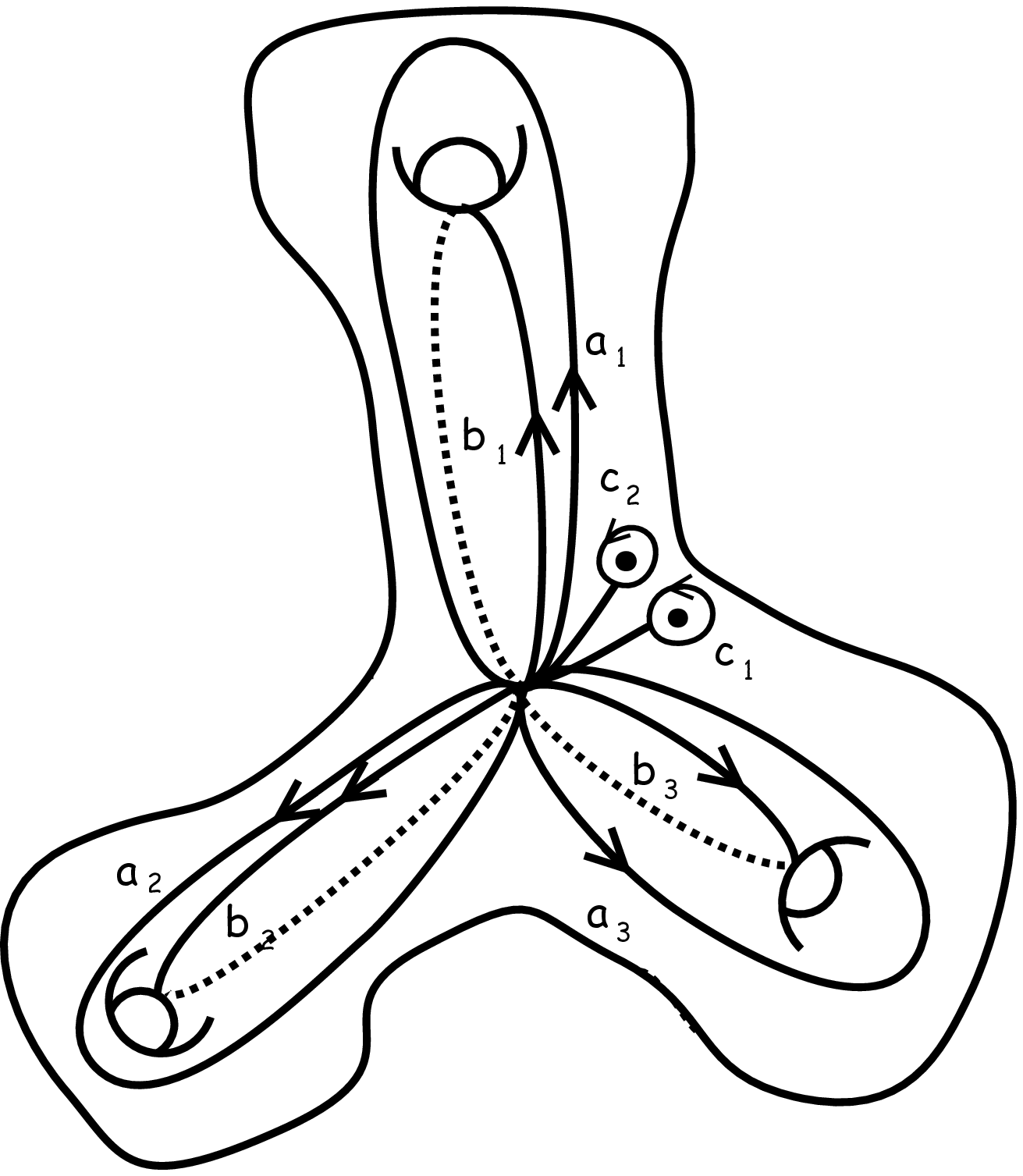, trim=0cm 9cm 0cm 0cm, width=6cm}}
\caption{} \label{presentation}
\end{figure}

Pour avoir une présentation du groupe fondamental de la surface $\Sigma_{g,p}$, on se donne un point base $x_0$ de $\Sigma_{g,p}$ et on peut choisir $2g+p$ lacets $(a_i)_{i=1...g}$, $(b_i)_{i=1...g}$ et $(c_j)_{j=1...p}$ comme sur la figure \ref{presentation}, de cet façon les lacets $(a_i)_{i=1...g}$ et $(b_i)_{i=1...g}$ font le tour des anses de $\Sigma_{g,p}$, alors que les lacets $(c_j)_{j=1..p}$ font le tour des pointes de $\Sigma_{g,p}$, ainsi, le groupe fondamental $\pi_1(\Sigma_{g,p})$ de $\Sigma_{g,p}$ admet alors la présentation:

$$\pi_1(\Sigma_{g,p}) = <a_1 ,... ,\, a_g ,\, b_1 ,... ,\, b_g ,\, c_1 ,... ,\, c_p \,|\, [a_1 ,\, b_1] \cdot \cdot \cdot [a_g ,\, b_g] c_1 \cdot \cdot \cdot c_p =1>$$

\underline{Dans la suite du texte on suppose que $p$ est un entier supérieur ou égal à 1.} Le groupe fondamental $\pi_1(\Sigma_{g,p})$ de $\Sigma_{g,p}$ est alors un groupe libre à $2g+p-1$ générateurs. Les lacets élémentaires (à orientation près) de $\Sigma_{g,p}$ sont donnés par les éléments $(c_j)_{j=1,..., p}$ de $\pi_1(\Sigma_{g,p})$, on gardera cette notation tout au long du texte.

Le lemme suivant qui est très classique  sera utile.

\begin{lemm}\label{varpara}
Soit $P=\{ \g \in \s \,|\, \g \textrm{ est parabolique} \}$, l'ensemble $P$ des éléments paraboliques de $\s$ est une sous-variété de $\s$ de dimension 6.
\end{lemm}

\begin{proof}
Commençons par remarquer que $P=\{ \g \in \s \,|\, (\g-1)^3 =0 \textrm{ et } (\g-1)^2 \neq 0 \}$. L'ensemble $P$ est donc un ouvert de Zariski du fermé de Zariski $\{ \g \in \s \,|\, (\g-1)^3 =0 \}$. Par conséquent $P$ est une variété algébrique. Le groupe $\s$ agit par conjugaison sur $P$ et cette action est transitive. Il vient que $P$ est une variété algébrique lisse. Le stabilisateur $Stab$ de la matrice:

$$
g=
\left(
\begin{array}{ccc}
1 & 1 & 0\\
0 & 1 & 1\\
0 & 0 & 1
\end{array}
\right)
\in P
$$

est le groupe:

$$
Stab=
\left\{
\begin{array}{c|l}
\left(
\begin{array}{ccc}
1 & a & b\\
0 & 1 & a\\
0 & 0 & 1
\end{array}
\right)
&
a,b \in \R
\end{array}
\right\}.
$$

L'ensemble $P$ est donc une sous-variété de $\s$ de dimension 8-2=6.
\end{proof}

On considère les quatre espaces suivants:

\begin{tabular}{lcl}
$H$   & $=$ & $\s^{2g+p}$\\

$H_P$ & $= $& $\s^{2g} \times P^p$\\
\end{tabular}

$H^{irr}=$
$
\left\{
\begin{tabular}{l|p{6cm}}
$(A_1 ,... ,\, A_g ,\, B_1 ,... ,\, B_g ,\, C_1 ,... ,\, C_p) \in H$ & Le groupe engendré par $(A_1 ,... ,\, A_g ,\, B_1 ,... ,\, B_g ,\, C_1 ,... ,\, C_p)$ est irréductible\\
\end{tabular}
\right\}
$

$H^{irr}_P =$
$
\left\{
\begin{tabular}{l|p{6cm}}
$(A_1 ,... ,\, A_g ,\, B_1 ,... ,\, B_g ,\, C_1 ,... ,\, C_p) \in H_P$ & Le groupe engendré par $(A_1 ,... ,\, A_g ,\, B_1 ,... ,\, B_g ,\, C_1 ,... ,\, C_p)$ est irréductible \\
\end{tabular}
\right\}
$

\par{
Ces espaces sont des sous-variétés. Pour le premier, ce fait est évident. Pour le second, le lemme \ref{varpara} rend ce fait évident. L'espace $H^{irr}$ (resp. $H^{irr}_P$) est un ouvert de $H$ (resp. $H_P$) par conséquent c'est une variété.
}
\\
\par{
On défini sur $H$ l'application différentiable $\mathrm{R} :H \rightarrow \s$ donné par:
$$R: (a_1 ,... ,\, a_g ,\, b_1 ,... ,\, b_g ,\, c_1 ,... ,\, c_p) \mapsto [a_1 ,\, b_1] \cdot \cdot \cdot [a_g ,\, b_g] c_1 \cdot \cdot \cdot c_p$$.

L'image réciproque de $\{ 1 \}$ par $R$ dans $H$ s'identifie naturellement avec l'espace des représentations du groupe fondamental de $\Sigma_{g,p}$ dans $\s$. L'identification est obtenue par l'évaluation d'une représentation $\rho:\pi_1(\Sigma_{g,p})\rightarrow \s$ sur les éléments $a_1 ,... ,\, a_g ,\, b_1 ,... ,\, b_g ,\, c_1 ,... ,\, c_p$.
}
\\
\par{
On dira qu'une représentation de $\Sigma_{g,p}$ dans $\s$ \emph{conserve les paraboliques} lorsque l'image des lacets $(c_j)_{j=1..p}$ est parabolique.
}
\\
\par{
Bien entendu, l'image réciproque de $\{ 1 \}$ par $R$ dans $H_P$ (resp. $H^{irr}$ resp. $H_P^{irr}$) s'identifie naturellement avec l'espace des représentations qui conservent les paraboliques (resp. irréductibles resp. irréductibles et qui conservent les paraboliques) du groupe fondamental de $\Sigma_{g,p}$.
}
\\
\par{
On notera $Hom_P= R_{|H_P}^{-1}\{1\} $ (resp. $Hom_P^{irr}= R_{|H_P^{irr}}^{-1}\{1\}$) l'ensemble des représentations (resp. irréductibles) du groupe fondamental de $\Sigma_{g,p}$ dans $\s$ qui conservent les paraboliques.
}
\\
\par{
Goldman a montré le lemme suivant grâce au calcul de Fox dans \cite{Gold3}.
}

\begin{lemm}\label{rang}
L'application $\mathrm{R} :H_P \rightarrow \s$ donné par:

$$R: (a_1 ,... ,\, a_g ,\, b_1 ,... ,\, b_g ,\, c_1 ,... ,\, c_p) \mapsto [a_1 ,\, b_1] \cdot \cdot \cdot [a_g ,\, b_g] c_1 \cdot \cdot \cdot c_p$$

\noindent  est une application différentiable dont le rang au point $(a_1 ,... ,\, a_g ,\, b_1 ,... ,\, b_g ,\, c_1 ,... ,\, c_p)$ est égale à la codimension du centralisateur du groupe engendré par $a_1 ,... ,\, a_g ,\, b_1 ,... ,\, b_g ,\, c_1 ,... ,\, c_p$ dans $\s$.
\end{lemm}

%\begin{proof}
%On note $\mathfrak{G}$ l'algèbre de Lie de $G$ et $\mathfrak{P}$ l'algèbre de Lie de $P$. Le calcul de Fox  montre que l'image de la différentielle $dR: \mathfrak{G}^{2g} \times \mathfrak{P}^p \rightarrow \mathfrak{G}$ est la somme:
%
%$$\Sigma_{i=1,...,g} Ad(\frac{\partial R}{\partial a_i})(\mathfrak{G}) + Ad(\frac{\partial R}{\partial b_i})(\mathfrak{G}) + \Sigma_{j=1,...,p} Ad(\frac{\partial R}{\partial c_j})(\mathfrak{P})$$
%
%On note
%
%\end{proof}

\begin{coro}
L'espace $Hom_P^{irr}$ est une variété de dimension $16g+6p-8$.
\end{coro}

\begin{proof}
Le lemme \ref{varpara}  montre l'espace $H_P$ est une variété de dimension $8 \times 2g + 6 \times p$. Le sous-ensemble $H^{irr}_P$ de $H_P$ est un ouvert et donc une sous-variété de $H_P$ de même dimension. Le lemme \ref{rang}  montre que le sous-ensemble $Hom_P^{irr} = R_{|H_P^{irr}}^{-1}\{1\}$ est une sous-variété de $H^{irr}_P$ de dimension $8 \times 2g + 6 \times p - 8 = 16g+6p-8$.
\end{proof}

On peut trouver une démonstration de Goldman de la proposition suivante dans \cite{Gold1}.

\begin{prop}
Le groupe $\s$ agit proprement et librement sur l'espace $H^{irr}$.
\end{prop}

Le corollaire suivant est à présent clair.

\begin{coro}\label{repvariette}
Le quotient $Rep_P^{irr}=Hom_P^{irr}/_{\s}$ est une variété de dimension $16g-16+6p$.
\end{coro}

On s'intéresse à l'espace suivant:

$$
\beta'_{g,p}=
\left\{
\begin{array}{l|l}
                                        & \rho \textrm{ est fidèle et discrète},\\
\rho:\pi_1(\Sigma_{g,p}) \rightarrow \s & \rho \textrm{ préserve un ouvert proprement convexe } \Omega_{\rho},\\
                                        & \forall j = 1,...,p, \,\, \rho(c_j) \textrm{ est parabolique},\\
                                        & \Omega_{\rho}/_{\rho(\pi_1(\Sigma_{g,p}))} \textrm{ est homéomorphe à la surface } \Sigma_{g,p}.
\end{array}
\right\}
$$

\par{
L'holonomie d'une structure projective proprement convexe  fournit une application de $\beta_f(\Sigma_{g,p}) \rightarrow \beta_{g,p} = \beta'_{g,p}/_{\s}$. Il est clair que cette application est continue et l'auteur a montré qu'elle était injective dans \cite{ludo}. Elle est même sujective. En effet, l'identification $\varphi: \Sigma_{g,p} \rightarrow \Omega_{\rho}/_{\rho(\pi_1(\Sigma_{g,p}))}$ fournit une représentation $\rho'$ de $\pi_1(\Sigma_{g,p})$ vers $\s$ qui préserve l'ouvert $\Omega_{\rho}$. Mais, l'automorphisme
$\rho^{-1} \circ \rho' : \pi_1(\Sigma_{g,p}) \rightarrow \pi_1(\Sigma_{g,p})$ n'a aucune raison d'être intérieur. Le théorème de Nielsen montre que quitte à composer $\varphi$ à la source par un homéomorphisme on peut supposer que $\rho^{-1} \circ \rho : \pi_1(\Sigma_{g,p}) \rightarrow \pi_1(\Sigma_{g,p})$ est intérieur.
}
\\
\par{
De plus, l'auteur a montré dans \cite{ludo} que $\beta_{g,p} \subset Rep_P^{irr}$. On note $Rep_P$ l'espace topologique quotient $Hom_P/_{\s}$. Il faut faire attention que contrairement à $Rep_P^{irr}$, la topologie de cet espace n'a rien d'évidente, en particulier elle n'a aucune raison d'être séparée. Le théorème suivant  montre que les structures projectives proprement convexe de volume fini sont des objets naturels.
}

\begin{theo}\label{comptheo}
L'application holonomie de  $\beta_f(\Sigma_{g,p})$ vers $Rep_P$ est un homéomorphisme sur son image $\beta_{g,p}$, en particulier $\beta_{g,p}$ est une composante connexe de $Rep_P$.
\end{theo}

La démonstration de ce théorème se déroule en deux parties. On commence par montrer que $\beta_{g,p}$ est un fermé de $Rep_P$. Ensuite, comme la topologie de $\beta_f(\Sigma_{g,p})$ est connue nous utiliserons le théorème d'invariance du domaine pour montrer que $\beta_{g,p}$ est un ouvert de  $Rep_P^{irr}$. Enfin, comme $R_P^{irr}$ est un ouvert de $Rep_P$, on obtiendra l'ouverture de $\beta_{g,p}$ dans $Rep_P$.

\subsection{Fermeture de $\beta'_{g,p}$}

Nous consacrons cette partie à la démonstration de la fermeture. Cet exposé est inspiré de l'article \cite{ChGo} qui traite le cas compact. Leur démonstration contient beaucoup d'idées mais ne gère pas le cas non compact.

\begin{prop}\label{ferm}
L'espace $\beta'_{g,p}$ est un fermé de $\textrm{Hom}(\pi_1(\Sigma_{g,p}), \s)$.
\end{prop}

\subsubsection{Lemmes préliminaires}

On aura besoin du lemme suivant dû à Goldman et Millson \cite{GoMi}:

\begin{lemm}[Goldman-Millson]\label{nil}
Soient $\G$ un groupe de type fini n'admettant pas de sous-groupe distingué nilpotent infini et $G$ un groupe de Lie, alors l'ensemble des morphismes fidèles et discrets de $\G$ vers $G$ est fermé dans l'ensemble des morphismes de $\G$ vers $G$.
\end{lemm}

On rappelle le lemme suivant que l'on a énoncé au début de ce texte.

\begin{lemm}\label{trsup}
Soit $\Omega$ un ouvert proprement convexe et $\g \in \Aut(\Omega)$ d'ordre infini alors $\textrm{Tr}(\g) \geqslant 3$.
\end{lemm}

\begin{lemm}\label{trinf}
Soit $\rho$ une représentation fidèle et discrète d'un groupe libre non abélien $\G$ dans $\s$, si $\rho$ n'est pas irréductible alors il existe un $\g \in \G$ tel que $\textrm{Tr}(\g) < 1$.
\end{lemm}

\begin{proof}
Supposons que la représentation $\rho$ n'est pas irréductible alors il existe une droite $D$ ou un plan $\Pi$ préservé par $\rho$. On note $G$ le stabilisateur dans $\s$ de cette droite ou de ce plan. Le groupe $G$ est de l'une des deux formes suivantes:
$$
\begin{array}{ccc}
\left(
\begin{array}{ccc}
* & * & *\\
0 & * & *\\
0 & * & *
\end{array}
\right)
&
\textrm{  ou  }
&
\left(
\begin{array}{ccc}
* & 0 & 0\\
* & * & *\\
* & * & *
\end{array}
\right)
\end{array}
$$

Dans les deux cas, on a un morphisme $\varphi:G \rightarrow \mathrm{SL_2(\mathbb{R})}$ tel que pour tout $\g \in [G,G]$ on a $\textrm{Tr}(\g) = 1 + \textrm{Tr}(\varphi(\g))$. Le sous-groupe $\rho([\G,\G])$ de $\s$ est un sous-groupe libre et discret d'un conjugué du groupe spécial affine de $\R^2$ ou $(\R^2)^*$. Par conséquent, la représentation $\varphi \circ \rho:\G\rightarrow \mathrm{SL_2(\mathbb{R})}$ restreinte à $[\G,\G]$ est fidèle et discrète. Le groupe $\varphi \circ \rho([\G,\G])$ est donc un sous-groupe libre non abélien et discret de $\mathrm{SL_2(\mathbb{R})}$. Le lemme \ref{trlib} qui suit conclut la démonstration.
\end{proof}

Ce lemme est démontré dans \cite{ChGo}.

\begin{lemm}\label{trlib}
Tout groupe libre non abélien et discret de $\mathrm{SL_2(\mathbb{R})}$ possède des éléments dont la trace est strictement négative.
\end{lemm}

Le lemme suivant est aussi démontré dans \cite{ChGo}.

\begin{lemm}\label{preserve}
Soient $(X,d)$ un espace métrique compact et une suite $(\g_n)_{n \in \N}$ d'homéomorphismes uniformément lipschitziens de X, on suppose que pour tout $n \in \N$, l'élément $\g_n$ préserve un fermé $F_n$. On suppose que les fermés $F_n$ convergent vers un fermé $F$ de $X$ pour la topologie de Hausdorff. Enfin, on suppose que la suite $(\g_n)_{n \in \N}$ converge uniformément vers un homéomorphisme $\g$ de $X$. Alors, $\g$ préserve $F$.
\end{lemm}

\begin{proof}
Soit $y \in F$, il faut montrer que $\g y \in F$. On se donne un réel $\varepsilon > 0$.

La suite $(\g_n)_{n \in \N}$ converge uniformément vers $\g$, par conséquent il existe $N_1 \in \N$ tel que:

$$\forall n > N_1,\, \forall x \in X,\, d(\g_n x, \g x)< \frac{\varepsilon}{2}$$

Les homéomorphismes $(\g_n)_{n \in \N}$ sont uniformément lipschitziens, par conséquent, il existe un $C > 0$ tel que:

$$\forall n \in \N,\, \forall x,\, y \in X,\, d(\g_n x, \g_n y)< C d(x,y)$$

La suite $F_n$ converge pour la topologie de Hausdorff vers le fermé $F$ par conséquent, il existe une suite $y_n \in F_n$ telle que $y_n$ converge vers $y$. Il existe donc $N_2 \in \N$ tel que:

$$\forall n > N_2,\, d(y_n,y)< \frac{\varepsilon}{2C}$$

On a donc:

$$
\begin{array}{ccl}
d(\g y,\g_n y_n) & \leqslant & d(\g y, \g_n y) + d(\g_n y, \g_n y_n)\\
                 & \leqslant & \frac{\varepsilon}{2} + C  \frac{\varepsilon}{2C} = \varepsilon
\end{array}
$$
Dès que $n$ est supérieur à $N_1$ et $N_2$. La suite $(\g_n y_n)_{n \in \N}$ de points de $F_n$ converge donc vers $\g y$. Le point $\g y$ est donc dans $F$.
\end{proof}

\subsubsection{Preuve de la fermeture}

\begin{lemm}\label{lastlast}
L'ensemble des représentations discrètes, fidèles, irréductibles et qui préservent un ouvert proprement convexe d'un groupe libre $\G$ non abélien dans $\s$ est fermé dans l'espace des représentations de $\G$ dans $\s$.
\end{lemm}

\begin{proof}
Soit une suite $\rho_n$ de représentations de $\G$ qui converge vers une représentation $\rho_{\infty}$ de $\G$. Le lemme \ref{nil} montre que $\rho_{\infty}$  est fidèle et discrète.

Commençons par montrer que la représentation $\rho_{\infty}$ est irréductible. Le lemme \ref{trsup} montre que pour tout $\g \in \G$ et tout $n \in \N$ on a Tr($\rho_n(\g)) \geqslant 3$. Par conséquent, comme la trace est une fonction continue, on a pour tout $\g \in \G$, Tr($\rho_{\infty}(\g)) \geqslant 3$. Le lemme \ref{trinf}  montre que la représentation $\rho_{\infty}$ est irréductible.

Montrons à présent que $\rho_{\infty}$ préserve un ouvert proprement convexe $\Omega_{\infty}$. Pour tout $n$, la représentation $\rho_n$ préserve un ouvert $\Omega_{\rho_n}$ proprement convexe de $\P$. On considère la 2-sphère euclidienne $\S^2$ et on note $\pi:\S^2 \rightarrow \P$ le revêtement universel canonique de $\P$. L'ouvert $\pi^{-1}(\Omega_{\rho_n})$ possède deux composantes connexes. On note $\Omega'_{\rho_n}$ l'une d'elles.

On munit $\S^2$ de la topologie de Hausdorff via sa métrique euclidienne canonique, l'ensemble des compacts de $\S^2$ forment un espace compact. Et, l'ensemble des fermés convexes de $\S^2$ est un fermé de cet espace. On peut donc supposer, quitte à extraire, que la suite des fermés convexes $\overline{\Omega'_{\rho_n}}$ converge vers un fermé convexe $C_{\infty}$ de $\S^2$.

Pour tout élément $\g \in \G$, la suite des $\rho_n(\g)$ est uniformémént lipschitzienne puisque cette suite est convergente dans $\s$. Le lemme \ref{preserve}  montre que $C_{\infty}$ est $\rho_{\infty}$-invariant. On a quatre possibilités pour le convexe fermé $C_{\infty}$:

\begin{itemize}
\item Le convexe $C_{\infty}$ est proprement convexe et d'intérieur non vide,

\item Le convexe $C_{\infty}$ est un segment,

\item Le convexe $C_{\infty}$ est un point,

\item Le convexe $C_{\infty}$ n'est pas proprement convexe.
\end{itemize}

On va montrer que seul le premier cas est possible. Pour cela remarquons que dans les trois autres cas il existe une droite $D$ de $\R^3$ ou un plan $\Pi$ de $\R^3$ qui est invariant sous $\rho_{\infty}$. Par conséquent dans ces trois cas, la représentation $\rho_{\infty}$ n'est pas irréductible. Ce qui est absurde.
\end{proof}

Pour montrer la proposition \ref{ferm}. Il ne nous reste plus qu'à montrer que le quotient $S_{\infty} = \Omega_{\infty}/_{\rho_{\infty}(\pi_1(\Sigma_{g,p}))}$ est homéomorphe à $\Sigma_{g,p}$. C'est le point où la démonstration diffère du cas compact, en effet dans le cas compact le groupe fondamental caractérise la topologie de la surface.

\begin{proof}[Démonstration de la proposition \ref{ferm}]
Soit une suite $\rho_n \in \beta'_{g,p}$ de représentations qui converge vers une représentation $\rho_{\infty}$. Le lemme \ref{lastlast} montre que $\rho_{\infty}$  est fidèle, discrète, irréductible et préserve un ouvert proprement convexe $\Omega_{\infty}$.

Pour cela, on se donne une triangulation idéale $\mathcal{T}$ de la surface topologique $S=\Sigma_{g,p}$. Le relèvement de $\mathcal{T}$  fournit une triangulation $\widetilde{\mathcal{T}}$ de la surface topologique $\widetilde{S}$. On note $dev_n:\widetilde{S} \rightarrow \Omega_n$ la développante associée à $\rho_n$.

Comme la développante $dev_n$ est un homéomorphisme sur un ouvert proprement convexe et l'holonomie de tous les bouts de la surface $S_n = \Omega_{n}/_{\rho_{n}(\pi_1(\Sigma_{g,p}))}$ est parabolique, l'image par $dev_n$ de tout relevé du bord d'un triangle topologique de $\widetilde{\mathcal{T}}$ converge en $+ \infty$ (resp.$- \infty$) vers un point de $\partial \Omega_n$ qui est fixé par un élément parabolique de $\G_n$. Ainsi, si $\lambda$ est un côté d'un des triangles de la triangulation $\widetilde{\mathcal{T}}$, on notera $\lambda_n = ]dev_n(\lambda)(-\infty),dev_n(\lambda)(+\infty)[$. De plus, les images par $dev_n$ de deux relevés quelconques $\lambda^1$, $\lambda^2$ qui bordent le même triangle de $\widetilde{\mathcal{T}}$ ont la même limite en $+ \infty$ ou bien $- \infty$. On a donc construit, pour tout $n \in \N$, grâce à $\widetilde{\mathcal{T}}$ une triangulation idéale et géodésique $\widetilde{\mathcal{T}_n}$ de l'ouvert $\Omega_n$ préservé par $\rho_n(\G)$. Si $\widetilde{T}$ désigne un triangle de $\widetilde{\mathcal{T}}$ alors on notera $\widetilde{T_n}$ le triangle correspondant dans la triangulation $\widetilde{\mathcal{T}_n}$ de l'ouvert $\Omega_n$. Nous allons montrer que la suite de triangulation $(\widetilde{\mathcal{T}_n})_{n \in \N}$ sous-converge vers une triangulation idéale et géodésique de $\Omega_{\infty}$ préservée par $\rho_{\infty}(\G)$.

La suite des ouverts proprement convexes $\Omega_n$ converge vers l'ouvert proprement convexe $\Omega_{\infty}$. Pour tout triangle topologique $T_i$ de $\mathcal{T}$, on \underline{fixe} un relevé de $\widetilde{T_i}$. Soit $i_0=1,..., N$, on se donne $\lambda$ un des 3 côtés du triangle $\widetilde{T_{i_0}} \subset \widetilde{S}$. Nous allons montrer que, quitte à extraire, la suite $(\lambda_n)_{n \in \N}$ converge vers un segment non trivial de $\Omega_{\infty}$. Il est évident, que quitte à extraire, la suite des segments fermés $(\overline{\lambda_n})_{n \in \N}$ converge vers un segment de $C_{\infty}$. Mais il n'est pas clair que ce segment n'est pas réduit à un point. Nous allons montrer que cette possibilité est absurde.

Supposons qu'il existe un segment $\widetilde{\lambda}$ tel que la suite de segment $\overline{\lambda_n}$ converge vers un point $p_{\infty}$ de $\partial \Omega_{\infty}$. Quitte à extraire, on peut alors supposer que l'une des deux composantes connexes $\widetilde{S_{moit}}$ de la surface $\widetilde{S}-\lambda$ vérifie que, pour tout triangle $\widetilde{T}$ de la triangulation $\widetilde{\mathcal{T}}$, si $\widetilde{T}$ est inclus dans $\widetilde{S_{moit}}$ alors $\widetilde{T_n}$ converge vers le point $p_{\infty}$.

Pour tout $i=1,...,N$, l'orbite sous $\G$ du triangle topologique $\widetilde{T_i} \subset \widetilde{S}$ intersecte l'ouvert $\widetilde{S_{moit}}$. Par conséquent, pour tout triangle topologique $\widetilde{T}$ de $\widetilde{\mathcal{T}}$ la suite $(\widetilde{T_n})_{n \in \N}$ converge vers un point de $\partial \Omega_{\infty}$. On considère le triangle topologique $\widetilde{T'}$ qui bordent $\lambda$ et qui n'est pas contenu dans $\widetilde{S_{moit}}$. La suite $(\widetilde{T'_n})_{n \in \N}$ converge vers un point, et ce point est $p_{\infty}$ puisque le segment $\overline{\lambda_n}$ converge vers le point $p_{\infty}$.

Pour montrer que ce raisonnement entraine que, pour tout triangle topologique $\widetilde{T}$ de notre triangulation $\widetilde{\mathcal{T}}$, la suite $(\widetilde{T_n})_{n \in \N}$ converge vers le point $p_{\infty}$, on peut introduire le graphe dual $\mathcal{G}$ de la triangulation $\widetilde{\mathcal{T}}$ sur $\widetilde{S}$. En faisant une récurrence sur la distance de graphe d'un triangle topologique $\widetilde{T}$ de la triangulation $\widetilde{\mathcal{T}}$ à $\widetilde{T'}$, on obtient facilement le résultat.

Par conséquent, comme la représentation $\rho_{\infty}$ est la limite des représentations $\rho_n$, on a que $\rho_{\infty}$ fixe le point $p_{\infty}$, ce qui est absurde puisque cette représentation est irréductible. On vient donc de montrer que pour tout $i=1,..., N$, et pour tout côté $\lambda$ du triangle $\widetilde{T_i} \subset \widetilde{S}$, la géodésique $\lambda_n$ de $\Omega_n$ converge vers une géodésique de $\Omega_{\infty}$. Par conséquent, la triangulation $(\widetilde{\mathcal{T}_n})_{n \in \N}$ sous-converge vers une triangulation de $\Omega_{\infty}$ préservée par $\rho_{\infty}(\G)$. La surface $S_{\infty}$ est donc homéomorphe à la surface $\Sigma_{g,p}$.
\end{proof}

\subsection{Conclusion}

On est à présent en mesure de montrer le théorème \ref{comptheo}. Pour cela, on rappelle le théorème de l'invariance du domaine du à Brouwer.

\begin{theo}[Brouwer, invariance du domaine]\label{brouwer}
Soient $M,N$ deux variétés de même dimension et $i:M \rightarrow N$ une injection continue. Si $M$ est connexe et l'image de $i$ est fermée alors $i$ est un homéomorphisme de $M$ vers une composante connexe de $N$.
\end{theo}

\begin{proof}[Démonstration du théorème \ref{comptheo}]
L'application naturelle de $\beta_f(\Sigma_{g,p})$ vers $R_P^{irr}$ est une injection continue. La proposition \ref{ferm} montre que son image est fermée, le théorème \ref{modulesansbord} montre que $\beta_f(\Sigma_{g,p})$ est une variété de dimension $16g-16+6p$, le corollaire \ref{repvariette} montre que $R_P^{irr}$ est une variété de dimension $16g-16+6p$. Le théorème \ref{brouwer} nous montre que l'holonomie est un homéomorphisme de $\beta_f(\Sigma_{g,p})$ vers son image $\beta_{g,p}$ dans $R_P^{irr}$. Par conséquent, $\beta_{g,p}$ est à la fois ouvert et fermé dans $R_P^{irr}$. Mais l'irréductibilité est une condition ouverte donc $R_P^{irr}$ est ouvert dans $R_P$, par suite $\beta_{g,p}$ est ouvert dans $R_P$. La proposition \ref{ferm} montre $\beta_{g,p}$ est aussi fermé dans $R_P$. Ce qui conclut la démonstration du théorème \ref{comptheo}.
\end{proof}

\backmatter
\bibliographystyle{alpha}

\end{document}